\documentclass{amsart}
\usepackage{amsfonts,amsmath,amssymb,amsthm,colonequals,nicefrac,bm,algorithmic}
\usepackage{graphicx,xcolor}
\usepackage{fullpage}
\usepackage{hyperref}
\usepackage[linesnumbered,algoruled,boxed]{algorithm2e}

\newcommand{\N}{{\mathbb N}}
\newcommand{\R}{{\mathbb R}}
\newcommand{\U}{{\mathbb U}}
\newcommand{\V}{{\mathbb V}}
\newcommand{\Vh}{{\V_h}}
\newcommand{\Uh}{{\mathbb U_h}}
\newcommand{\Uad}{{\U_\mathsf{ad}}}
\newcommand{\Uadh}{{\U_{\mathsf{ad}}^{h}}}
\newcommand{\Y}{{\mathbb Y}}
\newcommand{\bx}{{\bm x}}
\newcommand{\dx}{{\,\mathrm d\bx}}
\newcommand{\ua}{{u_\mathsf{a}}}
\newcommand{\ub}{{u_\mathsf{b}}}
\newcommand{\vida}{{\mathsf v _1^{\mathrm{id}}}}
\newcommand{\vidb}{{\mathsf v _2^{\mathrm{id}}}}
\newcommand{\J}{{\mathcal J}}

\newtheorem{theorem}{Theorem}[section]
\newtheorem{lemma}[theorem]{Lemma}
\newtheorem{proposition}[theorem]{Proposition}
\begin{document}
\title{Error estimates for a multiobjective optimal control of a pointwise tracking problem}
\thanks{This project has received funding by the Federal Ministry of Education and Research (BMBF) and the Baden-W\"urttemberg Ministry of Science as part of the Excellence Strategy of the German Federal and State Governments. In addition, the first author has been supported by ANID through FONDECYT postdoctoral project 3230126.}
%
\author{Francisco Fuica}
\address{Departamento de Matem\'atica y Ciencia de la Computaci\'on, Universidad de Santiago de Chile, Santiago, Chile.}
\address{Zukunftskolleg, Konstanz University, Universit\"atsstra{\ss}e 10, 78464 Konstanz, Germany.}
\email{francisco.fuica@usach.cl}
\author{Stefan Volkwein}
\address{Department of Mathematics and Statistics, Konstanz University, Universit\"atsstra{\ss}e 10, 78464 Konstanz, Germany.}
\email{stefan.volkwein@uni-konstanz.de}
\date{\today}
\begin{abstract} 
    We analyze a pointwise tracking multiobjective optimal control problem subject to the Poisson problem and bilateral control constraints.
    To approximate Pareto optimal points and the Pareto front numerically, we consider two different finite element-based scalarization techniques, namely the weighted-sum method and the reference point method, where in both methods many scalar-constrained optimization problems have to be solved.
    We prove a priori error estimates for both scalarizations.
    The underlying subproblems of either method are solved with a Barzilai-Borwein gradient method.
    Numerical experiments illustrate the accuracy of the proposed method.
\end{abstract}
%
%
\maketitle

\section{Introduction}

Multiobjective optimization (see, e.g., \cite{MR2143243,Mie98}) plays an increasingly important role in modern applications, where several criteria are often of equal importance. The task in multiobjective optimization and multiobjective optimal control is therefore to compute the set of optimal compromises (the Pareto set) between the conflicting objectives. The advances in algorithms and the increasing interest in Pareto optimal solutions have led to a wide range of new applications related to optimal and feedback control -- potentially with non-smoothness both on the level of the objectives or in the system dynamics; cf. \cite{BDPV17,MR4491072,SGMPV24}, for instance. This results in new challenges such as dealing with expensive models (e.g., governed by partial differential equations (PDEs)) and developing dedicated algorithms handling the non-smoothness. Since in contrast to single-objective optimization, the Pareto set generally consists of an infinite number of optimal solutions, the computational effort can quickly become challenging, which is particularly problematic when the objectives are costly to evaluate or when a solution has to be presented very quickly.

This article deals with the multiobjective optimal control of a pointwise tracking problem. Let $\Omega\subset\mathbb{R}^d$, with $d\in \{2, 3\}$, be an open, bounded and convex polytope with boundary $\partial\Omega$. We set $\U=L^2(\Omega)$ and $V=H^1_0(\Omega)$ supplied with their common inner products
\begin{align*}
    (\varphi,\psi)_{\U}=\int_\Omega\varphi\psi\dx\text{ for }\varphi,\psi\in\U\quad\text{and}\quad (\varphi,\psi)_V=\int_\Omega\nabla\varphi\cdot\nabla\psi\dx \text{ for }\varphi,\psi\in V.   
\end{align*}
The state and control Hilbert spaces are given by $\Y=V\cap H^2(\Omega)$ and $\U$, respectively, equipped with their usual topologies. Moreover, $\Omega_1=\{\bx^1_1,\ldots,\bx^{n_1}_1\}$ and $\Omega_2=\{\bx^1_2,\ldots,\bx^{n_2}_2\}$ are given two (non-empty) finite sets in $\Omega$, where $\Omega_1\cap\Omega_2= \emptyset$ is allowed. At the points in $\Omega_1$ and $\Omega_2$ we are given the desired states $y^i_k\in\R$ for $1\le i\le n_k$ and $1\le k\le 2$.

In this work, we will be interested in the study of the following bicriterial optimal control problem:
\begin{equation}
    \label{eq:multi_obj_ocp}
    \tag{$\mathbf P$}
    \min\J(y,u)\quad\text{subject to (s.t.)}\quad(y,u)\in\Y\times\Uad\text{ satisfies }-\Delta y=u\text{ in } \Omega, \quad y=0\text{ on }\partial\Omega,
\end{equation}
where the objective $\J:\Y\times\Uad\to\R^2$ is given as
\begin{align}
    \label{def:cost_funct}
    \J(y,u)=
    \begin{pmatrix}
        \J_{1}(y,u)\\
        \J_{2}(y,u)
    \end{pmatrix}
    \colonequals 
    \frac{1}{2}
    \begin{pmatrix}
    \sum\limits_{i=1}^{n_1}\left(y(\bx^i_1) - y_1^i\right)^2 +\lambda_{1}\,{\|u\|}_\U^{2}\\[1mm]
    \sum\limits_{i=1}^{n_2}\left(y(\bx^i_2) - y^i_2\right)^2 + \lambda_{2}\,{\|u\|}_\U^{2}
\end{pmatrix},
\end{align}
with $\lambda_1,\lambda_2>0$ and the non-empty, closed, bounded, and convex set of admissible controls is defined by
\begin{equation}
    \label{def:box_cons}
    \Uad\colonequals \{u \in\U\,:\,\ua(\bx) \le u(\bx) \le\ub(\bx) ~\text{ for almost all (f.a.a.) }\bx\in \Omega \}
\end{equation}
with $\ua,\ub\in\U$ satisfying $\ua\le\ub$ in $\Omega$ almost everywhere (a.e.).

Error estimates for finite element (FE) discretizations of a \emph{scalar-valued} version of the problem \eqref{eq:multi_obj_ocp} ($\J(y,u)=\J_{1}(y,u)\in\mathbb{R}$) have been previously derived in a number of works. 
For such a problem, in \cite{MR3449612}, the authors obtained a priori and a posteriori error estimates using continuous piecewise linear functions to approximate the optimal state and adjoint state variables and the so-called variational discretization for the optimal control variable.
The a priori estimate derived for the approximation error of the control behaves as $\mathcal{O}(h^{2-d/2})$ \cite[Theorem 3.2]{MR3449612}. 
Later, the authors of \cite{MR3523574} proposed a fully discrete scheme that approximates the optimal state, adjoint state, and control variables with continuous piecewise linear functions and proved, in two dimensions, a $\mathcal{O}(h)$ convergence rate for the approximation error of the control variable \cite[Theorem 5.1]{MR3523574}; error estimates for the variational discretization scheme were also analyzed \cite[Theorem 5.2]{MR3523574}. 
In \cite{MR3800041}, based on the theory of Muckenhoupt weights and weighted Sobolev spaces, the authors analyzed a numerical scheme that discretizes the control variable with piecewise constant functions and the state and adjoint state variables with continuous piecewise linear functions.
In this work, for the error approximation of the optimal control variable, the authors proved in two dimensions a nearly optimal convergence rate \cite[Theorem 4.3]{MR3800041} and in three dimensions a suboptimal convergence rate $\mathcal{O}(h^{1/2}|\log h|)$. 
This was later improved to $\mathcal{O}(h|\log h|)$ in \cite[Theorem 6.6]{MR3973329}. 
In addition, the authors of \cite{MR3973329} provide, in two dimensions, a $\mathcal{O}(h^2|\log h|^2)$ error estimate for the variational discretization \cite[Theorem 7.5]{MR3973329} and a post-processing scheme \cite[Theorem 7.12]{MR3973329}. 
For appropriate extensions of the aforementioned results to problems involving the Stokes system, a semilinear elliptic equation, and the stationary Navier-Stokes equations, the interested reader is referred to \cite{MR4438718,MR4013930,MR4218114,2023arXiv230914511F,MR4304887}.

In contrast to the previous advances, the analysis of approximation techniques for multiobjective optimal control problems is relatively scarce, with only a few works addressing it; see \cite{Ban17,Ban21}.
To our knowledge, this is the first work that considers approximation techniques for a pointwise tracking multiobjective optimal control problem.
This problem may be relevant in applications, where the state observations are performed at specific locations of the domain. One of the main challenges that arises by considering these observations in the cost function $\J$ is that the corresponding adjoint state variables $p_{k}$ ($k\in\{1,2\}$) possess reduced regularity properties: $p_{k} \in W_0^{1,\mathsf{p}}(\Omega) \setminus (V \cap C(\overline{\Omega}))$, where $\mathsf{p}<\nicefrac{d}{(d-1)}$.
This makes the analysis of a priori error estimates more involved.

We present the notion of \emph{Pareto optimality} and investigate two parameter-dependent methods to derive the set of all Pareto optimal points -- the so-called \emph{Pareto front} -- namely, the weighted-sum method (see, e.g., \cite[Chapter 3]{MR2143243} and \cite[pp. 78-85]{Mie98}) and the reference point method (see, e.g., \cite[pp.~164-170]{Mie98} and \cite{RBWSZ09}).
For each one of the underlying scalar problems associated with these scalarization methods, we derive the existence of a unique solution and first-order optimality conditions.
Moreover, we prove that solutions of these scalar problems are Pareto optimal for the multiobjective optimal control problem \eqref{eq:multi_obj_ocp}.
We consider a FE discretization for each scalar problem and approximate the solution to the state and adjoint equations with continuous piecewise linear functions, whereas the control variable is approximated with piecewise constant functions.
The main contribution of the present work is the proof of optimal ($d=2$) and quasi-optimal ($d=3$) a priori error estimates for the scalar problems associated with the weighted-sum method and the reference point method.

The paper is organized in the following manner:
In the remainder of this section, we introduce the notation that we will use throughout the manuscript.
In Section \ref{sec:MO-OCP}, we present a weak formulation for the bicriterial optimization problem \eqref{eq:multi_obj_ocp} and recall notions of Pareto optimality and Pareto stationarity.
Moreover, we present the weighted-sum method and the reference point method and explain how these scalarization methods can be used to characterize the Pareto stationary front. 
Suitable FE discretizations for the weighted-sum method and the reference point method are proposed in Section \ref{sec:FEM}; a priori error estimates are also proved.
Numerical experiments are presented in Section \ref{sec:num_exp}.


\textbf{Notation.} Throughout this work we use standard notation for Lebesgue and Sobolev spaces; cf. \cite{Eva08}, for instance. 
Let $\mathcal{X}$ and $\mathcal{Y}$ be Banach function spaces, we write $\mathcal{X}\hookrightarrow \mathcal{Y}$ to denote that $\mathcal{X}$ is continuously embedded in $\mathcal{Y}$. 
The norm of $\mathcal{X}$ is denoted by $\| \cdot \|_{\mathcal{X}}$. Given $p \in (1,\infty)$, we denote by $q$ its Hölder conjugate, that is, the real number such that $\nicefrac{1}{p}+\nicefrac{1}{q}=1$. In this work, vector inequalities must be understood componentwise.
More precisely, given $n\in \N$ and $\mathfrak{a},\mathfrak{b}\in \R^n$, we write $\mathfrak{a}\leq \mathfrak{b}$ if $\mathfrak{a}_{i} \leq \mathfrak{b}_{i}$ for all $i=1,\ldots,n$; similarly $\mathfrak{a}\geq \mathfrak{b}$.
We define the sets $\R^n_{\ge \mathfrak{a}}\colonequals\{ \mathfrak{b}\in\R^n: \mathfrak{b} \geq \mathfrak{a}\}$ and $\R^{n}_{\le \mathfrak{a}}\colonequals\{ \mathfrak{b}\in \mathbb{R}^{n} : \mathfrak{b} \le\mathfrak{a}\}$.
For convenience, we shall write $\R^{n}_{\le}\colonequals\R^{n}_{\le 0}$ and $\R^{n}_{\ge}\colonequals\R^{n}_{\ge 0}$. The relation $a \lesssim b$ indicates that $a \leq C b$, with a constant $C>0$ that depends neither on $a$, $b$ nor on the discretization parameters. The value of $C$ may change at each occurrence.

\section{The optimal control problem}\label{sec:MO-OCP}

In this section, we recall basic facts for the problem \eqref{eq:multi_obj_ocp}. 
We start by briefly studying the state equation and its solution.

\subsection{The state equation}

For $u\in\U$, we consider the following weak formulation for the state equation: Find $y\in V$ such that
\begin{equation}
    \label{eq:weak_Poisson}
    (y, \varphi)_V= (u,\varphi)_\U \quad \text{for all }\varphi\in V. 
\end{equation}
The Lax-Milgram lemma directly yields the existence of a unique solution $y\in V$ to \eqref{eq:weak_Poisson}. 
Moreover, the convexity of $\Omega$ implies that $y\in \Y$ holds; see, e.g., \cite[Theorems 3.2.1.2 and 4.3.1.4]{MR775683} for $d=2$ and \cite[Theorems 3.2.1.2]{MR775683} and \cite[Section 4.3.1]{MR2641539} for $d=3$.
Due to the compact embedding $H^{2}(\Omega)\hookrightarrow C(\overline{\Omega})$ \cite[Theorem 6.3, Part III]{MR2424078}, it follows that $y\in C(\overline{\Omega})$ holds. Hence, a point evaluation of $y\in\Y$ is meaningful.

\subsection{Weak formulation} 

We begin by defining the control-to-state map $S: \U\ni u \mapsto y \in \Y$, where $y=Su$ is the unique solution to \eqref{eq:weak_Poisson}.  
With $S$ at hand, we introduce the bicriterial reduced cost function $j: \U\to\R^2$ as $j(u)\colonequals\J(Su,u)$, that is,
\begin{equation}
    \label{def:reduced_cost}
    j(u)=
    \begin{pmatrix}
        j_1(u)\\
        j_2(u)
    \end{pmatrix}
    =
    \begin{pmatrix}
        \J_1(Su,u)\\
        \J_2(Su,u)
    \end{pmatrix}
    =\frac{1}{2}
    \begin{pmatrix}
        \sum\limits_{i=1}^{n_1}\left((Su)(\bx^i_1) - y_1^i\right)^2+\lambda_1\,{\|u\|}_\U^{2}\\
        \sum\limits_{i=1}^{n_2}\left((Su)(\bx^i_2) - y_2^i\right)^2 + \lambda_2\,{\|u\|}_\U^{2}
    \end{pmatrix}.
\end{equation}
We immediately note that $(Su)(\bx^i_k)=y(\bx^i_k)$ is meaningful for $1\le i\le n_k$ and $k\in\{1,2\}$. Hence, the objective $j$ is also well-defined. The associated reduced optimal control problem reads:
\begin{equation}
    \label{eq:weak_funct}
    \tag{$\hat{\mathbf P}$}
    \min j(u)\quad\text{s.t.}\quad u\in \Uad.
\end{equation}
Problem \eqref{eq:weak_funct} involves a vector-valued objective, which makes it necessary to introduce the concepts of \emph{order relation} and \emph{Pareto optimality} (see, e.g., \cite{MR2143243}): The point $\bar u\in\Uad$ is \emph{Pareto} optimal for \eqref{eq:weak_funct} if, for each $k\in\{1,2\}$, there is no other control $u\in\Uad\setminus \{\bar u\}$ satisfying $j_k(u)\leq j_k(\bar u)$ and $j_l(u)<j_l(\bar u)$ for at least one $l\in \{1,2\}$. We also define the \emph{Pareto set} and \emph{Pareto front} as
\begin{align*}
    \mathcal{P}_s\colonequals \{u\in\Uad\,:\,u \text{ is Pareto optimal}\}\subset\U   
    \quad\text{and}\quad
    \mathcal{P}_f\colonequals j(\mathcal{P}_s) \subset \R^2,
\end{align*}
respectively.

Note that both $j_1$ and $j_2$ in \eqref{def:reduced_cost} are bounded from below and strictly convex.
Therefore, similar arguments to those provided in the proof of \cite[Theorem 2.14]{Troltzsch} yield that the \emph{ideal vector} $\mathsf{v}^{\mathrm{id}}=(\mathsf{v}_{1}^{\mathrm{id}},\mathsf{v}_2^{\mathrm{id}})$ given as 
\begin{equation}
    \label{def:ideal_vector}
    \vida\colonequals \min\{j_1(u)\,:\,u\in\Uad\}\quad\text{and}\quad\vidb\colonequals \min\{j_2(u)\,:\,u\in\Uad\}
\end{equation}
is well-defined.

We now present an auxiliary result that establishes differentiability properties of $j_1$ and $j_2$.

\begin{proposition}[Differentiability of $j_1$ and $j_2$]
    \label{prop:diff_j1_j2}
    Let $u\in\U$. The objective functions $j_1$ and $j_2$, defined in \eqref{def:reduced_cost}, are Fr\'echet differentiable. Their derivatives at $u$ are given as
    \begin{align*}
        &j_k^{\prime}(u)w= \sum_{i=1}^{n_k}\left(\left((Su)(\bx_k^i)-y_k^i\right)(Sw)(\bx_k^i)\right)+ \lambda_k\,(u,w)_\U \quad\text{for any }w\in\U\text{ and for }k\in\{1,2\}.
    \end{align*}
\end{proposition}

\begin{proof}
    The fact that $j_1$ and $j_2$ are Fr\'echet differentiable is an immediate consequence of the linearity and boundedness of the control-to-state map $S$.
    The characterization of their derivatives is obtained by straightforward computations.
\end{proof}

Before we continue with our analysis, we introduce, given $u\in\Uad$ and $k\in\{1,2\}$, the \emph{adjoint state} variable $p_k\in W_0^{1,\mathsf{p}}(\Omega)$ for any $\mathsf{p}\in (1,\nicefrac{d}{d - 1})$ as the unique solution to the variational problem
\begin{align}
    \label{eq:adjoint_eq}
    (\varphi, p_k)_V= \sum_{i=1}^{n_k}\left(y(\bx_k^i)-y_k^i\right)\varphi(\bx_k^i) \quad \text{for all }\varphi\in W_0^{1,\mathsf{q}}(\Omega),
\end{align}
where $\nicefrac{1}{\mathsf{p}}+\nicefrac{1}{\mathsf{q}}=1$ and $y=Su$ hold. Note that $p_k$ is not necessarily in $V$. However, the inner product $(\varphi, p_k)_V$ is well defined for $\varphi\in W_0^{1,\mathsf{q}}(\Omega)$ and $p_k\in W_0^{1,\mathsf{p}}(\Omega)$.
The right-hand side in \eqref{eq:adjoint_eq} is well defined, as $\varphi\in W_0^{1,\mathsf{q}}(\Omega)\hookrightarrow C(\overline{\Omega})$ is satisfied for $\mathsf q>d$ \cite[Theorem 6.3, Part III]{MR2424078}.
Well-posedness of the problem \eqref{eq:adjoint_eq} follows from \cite[Theorem 2]{MR812624}.

\begin{lemma}[Auxiliary result]
    \label{lemma:aux_res_adj}
    Let $u\in\U$, $y=Su$, and $p_k\in W_0^{1,\mathsf{p}}(\Omega)$ for any $\mathsf{p}\in (\nicefrac{2d}{2 + d},\nicefrac{d}{d - 1})$ be the unique solution to \eqref{eq:adjoint_eq} for $k\in\{1,2\}$. Then,
    \begin{align}
        \label{eq:identity_adjoint}
        (p_k,w)_{\U}=\sum_{i=1}^{n_k}\left(y(\bx_k^i)-y_k^i\right)(Sw)(\bx_k^i)\quad \text{for all }w\in \U.
    \end{align}
\end{lemma}

\begin{proof}
   Let $k\in\{1,2\}$ be chosen. For $\mathsf{q} \in (d,\nicefrac{2d}{d-2})$ we have $\Y\hookrightarrow W_{0}^{1,\mathsf{q}}(\Omega)$; see, e.g., \cite[Theorem 4.12]{MR2424078}. Thus, it follows that $Sw\in W_{0}^{1,\mathsf{q}}(\Omega)$ for every $w\in \U$ so that we can replace $\varphi=Sw$ in \eqref{eq:adjoint_eq} to obtain 
    \begin{align}
        \label{eq:identity_point_varphi}
        (Sw, p_k)_V 
        =
        \sum_{i=1}^{n_k} \left(y(\bx_k^i) - y_k^i\right)(Sw)(\bx_k^i).
    \end{align}
    Notice that $p_k\not\in V$, so that we cannot choose $\varphi=p_k$ in \eqref{eq:weak_Poisson}. Let $\{p_k^n\}_{n\in \N}\subset C_0^{\infty}(\Omega)$ such that $p_k^n\to p_k$ in $W_0^{1,\mathsf{p}}(\Omega)$ with $\mathsf{p}<\nicefrac{d}{d-1}$. Set, for $n\in\N$, $\varphi=p_k^n$ and $u=w$ in \eqref{eq:weak_Poisson}. 
    This results in
    \begin{equation}
        \label{eq:identity_adj_seq}
        (Sw, p_k^n)_{V}=(w,p_k^n)_{\U}. 
    \end{equation}
    Since $p_k^n\to p_k$ in $W_0^{1,\mathsf{p}}(\Omega)\hookrightarrow \U$, we observe that $|(Sw, p_k^n - p_k)_{V}| \to 0$ and $|(w,p_k^n - p_k)_{\U}| \to 0$ as $n\to \infty$.
Hence, passing to the limit in \eqref{eq:identity_adj_seq} yields $(Sw, p_k)_{V} = (w, p_k)_{\U}$, which concludes the proof.
\end{proof}

\subsection{Scalarization techniques}\label{sec:scala_tech}

The search of Pareto optimal points will be carried out using two methods: the weighted-sum method and the (Euclidean) reference point method (see, e.g., \cite[Section 4.3]{MR3950723} and \cite[Section 3]{MR3609766}, respectively).
In what follows, we describe both methods and present first-order optimality conditions for their corresponding underlying scalar problems.

\subsubsection{Weighted-sum method (WSM)}\label{sec:WSM}

For weights $\alpha_{1}, \alpha_{2} > 0$ with $\alpha_1 + \alpha_2 = 1$, we set $\alpha\colonequals(\alpha_{1},\alpha_{2})$ and define the (scalar) function
\begin{align*}
    W_{\alpha}(u)\colonequals\alpha_{1}j_{1}(u)+\alpha_2 j_2(u)\quad\text{for every }u\in \U.
\end{align*}
Then, the scalar-valued problem
\begin{equation}
    \label{eq:WSM_problem}
    \tag{$\mathbf{\hat P}_\alpha$}
    \min W_{\alpha}(u)\quad\text{s.t.}\quad u\in\Uad
\end{equation}
is considered.
The latter is denoted as the weighted-sum problem (with positive weights $\alpha_{1}$ and $\alpha_{2}$) corresponding to \eqref{eq:weak_funct}. The weighted-sum method is based on solving problem \eqref{eq:WSM_problem} for varying $\alpha$. 

In the next result, we prove the existence of solutions for \eqref{eq:WSM_problem}.

\begin{proposition}[Existence of a solution (WSM)]
    \label{Prop:WSM}
    Let $\alpha =(\alpha_{1},\alpha_{2})\in \mathbb{R}^2_{>}$ such that $\alpha_1+\alpha_2=1$. Then, the scalar-valued optimal control problem \eqref{eq:WSM_problem} has a unique solution $\bar u_\alpha \in \Uad$, which is Pareto optimal for problem \eqref{eq:weak_funct}.
\end{proposition}

\begin{proof}
     The fact that $j_{1}$ and $j_2$ are strictly convex and weakly lower semicontinuous, in combination with the fact that $\mathbb{U}_{ad}$ is weakly sequentially compact, yields the existence of a unique optimal control $\bar u_\alpha $ for \eqref{eq:WSM_problem}.

    To prove the second statement of the proposition, we proceed by contradiction and assume that $\bar u_\alpha $ is not Pareto optimal for \eqref{eq:weak_funct}. 
    Then, there exists $\hat{u}\in \mathbb{U}_{ad}$ satisfying $j(\hat{u})\leq j(\bar u_\alpha )$ and $j_k(\hat{u}) < j_k(\bar u_\alpha )$ for some $k\in\{1,2\}$.  Hence, $\alpha_{k}j_k(\hat{u}) < \alpha_{k}j_k(\bar u_\alpha )$ ($\alpha_{k}>0$), which implies that $W_{\alpha}(\hat{u}) < W_{\alpha}(\bar u_\alpha )$. This contradicts the optimality of $\bar u_\alpha $ and concludes the proof.
\end{proof}

In view of Proposition \ref{prop:diff_j1_j2} and Lemma \ref{lemma:aux_res_adj}, we infer that
\begin{align*}
    W^{\prime}_{\alpha}(u)w = & \, \sum_{k=1}^2\alpha_k\left(\sum_{i=1}^{n_k}((Su)(\bx_{k}^i) - y_{k}^i)(Sw)(\bx_k^{i}) + \lambda_k(u,w)_{\U}\right) \\
    = & \, \sum_{k=1}^2\alpha_k(p_{k} + \lambda_k u,w)_{\U} \quad \text{for any } w\in \U.
\end{align*}
With this characterization at hand, we state the first-order sufficient optimality conditions for the WSM whose proof follows by standard arguments; see, e.g., \cite[Lemma 2.21]{Troltzsch}.

\begin{theorem}[First-order optimality condition (WSM)]
    Let $\alpha=(\alpha_{1},\alpha_{2})\in \mathbb{R}^2_{>}$ be such that $\alpha_1 + \alpha_2 = 1$. If $\bar u_{\alpha}\in \Uad$ is an optimal solution to \eqref{eq:WSM_problem}, then we have
    \begin{equation}
        \label{eq:var_ineq_WSM}
        \sum_{k=1}^2\alpha_k(\bar{p}_{\alpha,k} + \lambda_k \bar u_\alpha , u - \bar u_\alpha )_{\U} \geq 0 \quad \text{for all }u\in\Uad,
    \end{equation}
    where $\bar{p}_{\alpha,k}$ denotes the unique solution to \eqref{eq:adjoint_eq} with $\bar{y}_{\alpha}=S\bar u_\alpha $.
\end{theorem}

Define $\bar{\mathfrak{p}}_{\alpha} = \alpha_1\bar{p}_{\alpha,1} +\alpha_2\bar{p}_{\alpha,2}$ and $\Lambda_{\alpha} = \alpha_1\lambda_1+\alpha_2\lambda_2$.
From \eqref{eq:var_ineq_WSM} (see, e.g., \cite[Theorem 2.28]{Troltzsch}) we infer that 
\begin{align}
\label{eq:projection_control_WSM}
    \bar u_{\alpha}(\bx) = \min\left\{\ub(\bx),\max\left\{\ua(\bx), - \Lambda_{\alpha}^{-1} \bar{\mathfrak{p}}_{\alpha}(\bx)\right\}\right\}\quad \text{f.a.a. } \bx \in \Omega.
\end{align}
We note that $\bar{\mathfrak{p}}_{\alpha}\in W_0^{1,\mathsf{p}}(\Omega)$, for any $\mathsf{p}\in (1,\nicefrac{d}{d - 1})$, solves
\begin{align*}
    (\varphi, \bar{\mathfrak{p}}_{\alpha})_V = \sum_{k=1}^{2}\alpha_{k}\sum_{i=1}^{n_k}\left(\bar{y}_{\alpha}(\bx_k^i)-y_k^i\right)\varphi(\bx_k^i) \quad \text{for all }\varphi\in W_0^{1,\mathsf{q}}(\Omega),
\end{align*}
where $\nicefrac{1}{\mathsf{p}}+\nicefrac{1}{\mathsf{q}}=1$ and $\bar{y}_{\alpha}=S\bar{u}_{\alpha}$ hold.
Hence, if $\ua,\ub\in H^{1}(\Omega)$, then, in view of \eqref{eq:projection_control_WSM} and \cite[Proposition 2.5]{MR3973329} we conclude that $\bar{u}_{\alpha} \in H^{1}(\Omega)$.
Moreover, if $\ua,\ub$ are just real constants, then \cite[Lemma 6.1]{MR3973329} immediately yields $\bar{u}_{\alpha}\in W^{1,\infty}(\Omega)$.

\subsubsection{Reference point method (RPM)}\label{sec:RPM}

Given a \emph{reference point} $\zeta = (\zeta_1,\zeta_2) \in \mathcal{P}_{f}+\mathbb{R}^2_{\leq}$ we introduce the distance function $R_{\zeta}:\U\to \mathbb{R}$ by
\begin{align*}
   R_{\zeta}(u)\colonequals\frac{1}{2}\,{\|j(u)-\zeta\|}_{\R^2}^2=\frac{1}{2}\left( (j_1(u) - \zeta_1)^2 + (j_2(u) - \zeta_2)^2\right)\quad \text{for }u\in\U,
\end{align*}
which measures the Euclidean distance between $j(u)$ and the reference point $\zeta$.  Moreover, if $\zeta \leq \mathsf{v}^{\mathrm{id}}$ with $\mathsf{v}^{\mathrm{id}}$ defined in \eqref{def:ideal_vector}, then the mapping $R_{\zeta}$ is strictly convex since it involves the composition between a strictly convex nondecreasing function and a strictly convex function.

The idea behind the reference point method is that by finding a point $u$ such that $j(u)$ approximates $\zeta$ as best as possible, we shall obtain a Pareto optimal point for \eqref{eq:weak_funct}. Therefore, we have to solve the (Euclidean) \emph{reference point problem}
\begin{equation}
    \label{eq:ref_point_prob}
    \tag{$\mathbf{\hat P}_\zeta$}
    \min R_{\zeta}(u)\quad\text{s.t.}\quad u\in\Uad,
\end{equation}
which is a scalar-valued minimization problem.

\begin{proposition}[Existence of a solution (RPM)]
    Let $\zeta = (\zeta_1,\zeta_2) \in \mathcal{P}_{f}+\mathbb{R}^2_{\leq}$ such that $\zeta \leq \mathsf{v}^{\mathrm{id}}$. Then, the scalar-valued optimal control problem \eqref{eq:ref_point_prob} has a unique solution $\bar{u}_{\zeta}\in \Uad$, which is Pareto optimal for problem \eqref{eq:weak_funct}.
\end{proposition}

\begin{proof}
    We can argue as in the proof of Proposition~\ref{Prop:WSM}. For brevity, we skip the details.
\end{proof}

For $u\in\U$ the use of Proposition \ref{prop:diff_j1_j2} and Lemma \ref{lemma:aux_res_adj} reveal that
\begin{align*}
    R^{\prime}_{\zeta}(u)w&=\sum_{k=1}^2(j_k(u) - \zeta_k)\left(\sum_{i=1}^{n_k}(Sw)(\bx_k^i)\left((Su)(\bx_k^i) - y_k^i\right) + \lambda_k(u,w)_\U\right) \\
    &= \sum_{k=1}^2(j_k(u) - \zeta_k)(p_k + \lambda_k u,w)_\U \quad \text{for any } w\in\U,
\end{align*}
from which we derive the following standard result.

\begin{theorem}[First-order optimality condition (RPM)]
    Let $\zeta=(\zeta_{1},\zeta_{2}) \in \mathcal{P}_f + \mathbb{R}_{\leq}^{2}$. If $\bar{u}_{\zeta}\in \Uad$ is an optimal solution to \eqref{eq:ref_point_prob}, then 
    \begin{equation}\label{eq:var_ineq_RPM}
        \sum_{k=1}^{2}(j_{k}(\bar{u}_{\zeta}) - \zeta_{k})(\bar{p}_{\zeta,k} + \lambda_k \bar{u}_{\zeta}, u - \bar{u}_{\zeta})_\U \geq 0 \quad \text{for all }u \in \Uad,
    \end{equation}
    where $\bar{p}_{\zeta,k}$ denotes the unique solution to \eqref{eq:adjoint_eq} with $\bar{y}_{\zeta}=S\bar{u}_{\zeta}$.
\end{theorem}

As in the Weighted-Sum Method (see \eqref{eq:projection_control_WSM}), we obtain a characterization for the optimal control $\bar{u}_{\zeta}$:
\begin{align*}
    \bar u_{\zeta}(\bx) = \min\{\ub,\max\{\ua, - \Lambda_{\zeta}^{-1} \bar{\mathfrak{p}}_{\zeta}(\bx) \}\}   \quad \text{f.a.a. } \bx \in \Omega,
\end{align*}
where
\begin{align*}
    \bar{\mathfrak{p}}_{\zeta} = (j_{1}(\bar{u}_{\zeta}) - \zeta_{1})\bar{p}_{\zeta,1} + (j_{2}(\bar{u}_{\zeta}) - \zeta_{2})\bar{p}_{\zeta,2},\qquad\Lambda_{\zeta} = \lambda_1(j_{1}(\bar{u}_{\zeta}) - \zeta_{1}) + \lambda_2(j_{2}(\bar{u}_{\zeta}) - \zeta_{2}).
\end{align*}
In particular, if $\ua,\ub$ are real constants, then $\bar{u}_{\zeta}\in W^{1,\infty}(\Omega)$ (see \cite[Lemma 6.1]{MR3973329}).

\section{The FE approximation}\label{sec:FEM}

In this section, we present suitable FE discretizations for the scalarization techniques proposed in Section~\ref{sec:scala_tech}. Moreover, we prove error estimates for the error committed when approximating solutions of the scalarized problems presented in the previous section.

Let us first introduce the discrete setting in which we will operate \cite{MR2373954,MR2050138}.
We denote by $\mathcal{T}_h = \{ T \}$ a conforming partition, or mesh, of $\overline{\Omega}$ into closed simplices $T$ of size $h_T\colonequals\text{diam}(T)$. 
Define $h\colonequals\max_{ T \in \mathcal{T}_h} h_T$. 
We denote by $\mathbb{T} = \{\mathcal{T}_h \}_{h>0}$ a collection of conforming and quasi-uniform meshes $\mathcal{T}_h$. Given a mesh $\mathcal{T}_{h} \in \mathbb{T}$, we define the FE space of continuous piecewise polynomials of degree one that vanish on the boundary as
\begin{equation}
    \label{def:piecewise_linear_set}
    \Vh\colonequals\left\{v^h\in C(\overline{\Omega}): v^h|_T\in \mathbb{P}_{1}(T)\text{ for all }T\in \mathcal{T}_{h}\right\}\cap V
\end{equation}
endowed with the inner product $(\cdot\,,\cdot)_V$. Let us also introduce the space
\begin{equation}
    \label{def:piecewise_constant_set}
    \Uh\colonequals\left\{u^h\in L^{\infty}(\Omega): u^h|_T\in \mathbb{P}_{0}(T)\text{ for all }T\in \mathcal{T}_{h}\right\}\subset\U.
\end{equation}

The FE-Galerkin approximation to \eqref{eq:weak_Poisson} is given as follows: Find $y^h\in \Vh$ such that
\begin{equation}
    \label{eq:discrete_Poisson}
    (y^h,\varphi^h)_V = (u,\varphi^h)_\U \quad \text{for all } \varphi^h\in \Vh. 
\end{equation}
We introduce the discrete control-to-state map $S_h: \U \ni u \mapsto y^h\in \Vh$, where $y^h=S_hu$ is the unique solution to \eqref{eq:discrete_Poisson}. 
The discrete function $j^h:\U\to\R^2$ is defined by
\begin{align}
    \label{def:reduced_cost_discrete}
    j^h(u)=
    \begin{pmatrix}
        j_{1}^{h}(u) \\
        j_{2}^{h}(u)
    \end{pmatrix}
    \colonequals
    \frac{1}{2}
    \begin{pmatrix}
        \sum\limits_{i=1}^{n_1}\left((S_h u)(\bx^i_1) - y_1^i\right)^2+\lambda_1\,{\|u\|}_\U^{2}\\
        \sum\limits_{i=1}^{n_2}\left((S_h u)(\bx^i_2) - y_2^i\right)^2 + \lambda_2\,{\|u\|}_\U^{2}
    \end{pmatrix}.
\end{align}
Finally, given $u\in\U$ and $k\in\{1,2\}$, we define $p_{k}^{h}\in\Vh$ as the unique solution to
\begin{align}
    \label{eq:adjoint_eq_discrete}
    (\varphi^h,p_{k}^{h})_V=\sum_{i=1}^{n_k}\left(y^h(\bx^i_k) - y_k^i\right)\varphi^h(\bx^i_k)\quad\text{for all }\varphi^h\in \Vh,
\end{align}
where $y^h=S_h u$.

In what follows, we shall introduce FE discretizations for the scalar problems \eqref{eq:WSM_problem} and \eqref{eq:ref_point_prob}, introduced in Sections \ref{sec:WSM} and \ref{sec:RPM}, respectively.

\subsection{Discretization of WSM}

For weights $\alpha_{1}, \alpha_{2} > 0$ with $\alpha_1 + \alpha_2 = 1$, we set $\alpha\colonequals(\alpha_{1},\alpha_{2})$ and define the function
\begin{align*}
    W_{\alpha}^{h}(u)\colonequals\alpha_{1}j_{1}^{h}(u)+ \alpha_2 j_{2}^{h}(u)\quad\text{for }u\in\U.    
\end{align*}
The proposed discrete version of the scalar-valued problem \eqref{eq:WSM_problem} thus reads
\begin{equation}
    \label{eq:WSM_problem_discrete}
    \tag{$\mathbf{\hat P}_{\alpha}^{h}$}
    \min W_{\alpha}^{h}(u^h)\quad\text{s.t.}\quad u^h\in\Uadh.
\end{equation}
As in the continuous case, we obtain a characterization for the derivative of $W_{\alpha}^{h}$:
\begin{align*}
    (W_{\alpha}^{h})^{\prime}(u^h)w^h=  \sum_{k=1}^2\alpha_k(p_{k}^{h} + \lambda_k u^h,w^h)_\U \quad \text{for all }w^h\in \Uh.
\end{align*}

We establish the existence of solutions for \eqref{eq:WSM_problem_discrete} and first-order optimality conditions.

\begin{proposition}[Discrete WSM]\label{prop:discrete_WSM}
    Let $\alpha =(\alpha_{1},\alpha_{2})\in\R^2_{>}$ such that $\alpha_1 + \alpha_2 = 1$. Then, the scalar-valued optimal control problem \eqref{eq:WSM_problem_discrete} has a unique solution $\bar{u}_{\alpha}^{h}\in \Uadh$. Moreover, $\bar{u}_{\alpha}^{h}$ is an optimal solution to \eqref{eq:WSM_problem_discrete} if and only if
    \begin{equation}
        \label{eq:discrete_var_ineq_WSM}
        \sum_{k=1}^2\alpha_k(\bar{p}_{\alpha,k}^{h} + \lambda_k \bar{u}_{\alpha}^{h}, u^h-\bar{u}_{\alpha}^{h})_\U \geq 0 \quad \text{for all }u^h \in \Uadh,
    \end{equation}
    where $\bar{p}_{\alpha,k}^{h}$ denotes the unique solution to \eqref{eq:adjoint_eq_discrete} with $\bar{y}_{\alpha}^{h}=S_h\bar{u}_{\alpha}^{h}$.
\end{proposition}

\begin{proof}
    Since $j_{1}^{h}$ and $j_2^{h}$ are continuous and $\Uadh$ is compact, Weierstra{\ss} theorem immediately yields the existence of a solution. Uniqueness follows from the strict convexity of $j_{1}^{h}$ and $j_2^{h}$. The variational inequality \eqref{eq:discrete_var_ineq_WSM} follows as in the continuous case.
\end{proof}

We describe the algorithm proposed to solve the problem \eqref{eq:WSM_problem_discrete}, which is based on a Barzilai-Borwein gradient method \cite{MR967848}; see also \cite[Algorithm 1]{MR4110640}.
\begin{algorithm}[ht]
	\caption{Solving problem \eqref{eq:WSM_problem_discrete} with $\alpha$ fixed}
    \label{alg:discrete_WSM}
	\begin{algorithmic}[1]
        \REQUIRE{Problem data, $\alpha=(\alpha_1,\alpha_2)\in\R^2_{>}$, $tol=1$, and $u^{h,-1},u^{h,0}\in \Uadh$ with $u^{h,-1}\neq u^{h,0}$;}
        \STATE{Set $l=0$;}
        \WHILE{$tol > 10^{-8}$}
            \STATE{Obtain $y^{h,l}\in\Vh$ by solving \eqref{eq:discrete_Poisson} with $u=u^{h,l}$;}
            \STATE{Obtain $p_{1}^{h,l},p_{2}^{h,l}\in \Vh$ by solving \eqref{eq:adjoint_eq_discrete} with $y^h=y^{h,l}$;}
            \STATE{Compute step size
            \begin{align*}
                t_l = \frac{((W_{\alpha}^{h})^{\prime}(u^{h,l}) - (W_{\alpha}^{h})^{\prime}(u^{h,{l-1}}),(W_{\alpha}^{h})^{\prime}(u^{h,l}) - (W_{\alpha}^{h})^{\prime}(u^{h,{l-1}}))_\U}{((W_{\alpha}^{h})^{\prime}(u^{h,l})-(W_{\alpha,h})^{\prime}(u^{h,{l-1}}), u^{h,l} - u^{h,{l-1}})_\U};
            \end{align*}}
            \STATE{Set $u^{h,{l+1}}= \min\{\ub,\max\{\ua,u^{h,l}- \nicefrac{1}{t_l}(W_{\alpha}^{h})^{\prime}(u^{h,l})\}\}$ and $u_{\mathrm{ref}}= \min\{\ub,\max\{\ua,u^{h,l}- (W_{\alpha}^{h})^{\prime}(u^{h,l})\}\}$ ;} 
            \STATE{Set $tol = \|u^{h,{l+1}} - u_{\mathrm{ref}}\|_{\U}$} and $l=l+1$;\footnotemark
        \ENDWHILE
        \RETURN{Optimal solution $\bar{u}_{\alpha}^{h}$ and its objective value $j^h(\bar{u}_{\alpha}^{h})$.}
	\end{algorithmic}
\end{algorithm}

In Algorithm \ref{alg:WSM}, we show how we compute the discrete approximation of the set of Pareto points and the corresponding front.
\begin{algorithm}[ht]
    \label{alg:WSM}
    \caption{Weighted-sum method}
    \begin{algorithmic}[1]
        \REQUIRE{Number $\ell_{\mathrm{max}}\in \mathbb{N}$ of stationary points, problem data, and $0<\varepsilon \ll 1$:}
        \STATE{Set $\mathcal{P}_{s}^h = \mathcal{P}_f^h = \emptyset$;}
        \FOR{$\ell=1,\ldots,\ell_\mathsf{max}$}
            \STATE{Set $\alpha_2=\varepsilon+\nicefrac{(\ell-1)}{(\ell_\mathsf{max}-1)}$;}
            \STATE{Solve problem \eqref{eq:WSM_problem_discrete} using Algorithm~\ref{alg:discrete_WSM} with weight $\alpha(\ell)=(1-\alpha_2,\alpha_2)$, save solution $\bar{u}_{\alpha(\ell)}^{h}$ and its objective value $j^h(\bar{u}_{\alpha(\ell)}^{h})$;}
            \STATE{Set $\mathcal{P}_{s}^h=\mathcal{P}_{s}^h \cup \{ \bar{u}_{\alpha(\ell)}^{h}\}$ and $\mathcal{P}_{f}^h=\mathcal{P}_{f}^h \cup \{ j^h(\bar{u}_{\alpha(\ell)}^{h})\}$;}
        \ENDFOR
        \RETURN{Discrete approximations $\mathcal{P}_{s}^h$ and $\mathcal{P}_f^h$ of Pareto stationary points and front, respectively.}
    \end{algorithmic}
\end{algorithm}

\subsubsection{Error estimates for WSM}

Let $\alpha =(\alpha_{1},\alpha_{2})\in \R^2_{>}$ be fixed with $\alpha_1 + \alpha_2 = 1$.
In this section, we prove convergence rates for $\|\bar u_\alpha  - \bar{u}_{\alpha}^{h}\|_\U$, where $\bar u_\alpha $ and $\bar{u}_{\alpha}^{h}$ denote the unique solutions to \eqref{eq:WSM_problem} and \eqref{eq:WSM_problem_discrete}, respectively.

We start by introducing the discrete auxiliary variable $\hat y^h \in \mathbb{V}_{h}$ as the unique solution to 
\begin{align}
    \label{def:hat_y}
    (\hat y^h,\varphi^h)_V= (\bar u_\alpha,\varphi^h)_\U \quad \text{for all }\varphi^h\in \mathbb{V}_{h}.
\end{align}
We also define, for $k\in\{1,2\}$, the variable $\hat{p}_{k}^{h}\in \Vh$ solution to
\begin{align}
    \label{def:hat_varphi}
    (\varphi^h,\hat p_{k}^{h})_V=\sum_{i=1}^{n_k}\left(\hat y^h(\bx_k^i) - y_k^i\right)\varphi^h(\bx_k^i) \quad \text{for all }\varphi^h \in \Vh.
\end{align}

\begin{proposition}[Auxiliary estimate]
    \label{prop:aux_estimate_hat_tilde}
    Let $\bar{p}_{\alpha,k}$, with $k\in\{1,2\}$, be the unique solution to \eqref{eq:adjoint_eq} associated with the optimal solution $\bar u_\alpha $ to \eqref{eq:WSM_problem}, and let $\hat p_{k}^{h}$, with $k\in\{1,2\}$, be the unique solution \eqref{def:hat_varphi}.
    Assume that $\ua,\ub$ are real constants.
    Then, there exists an $h_*>0$ such that
    \begin{equation*}
        \sum_{k=1}^2\|\bar{p}_{\alpha,k}-\hat p_{k}^{h}\|_{L^{1}(\Omega)} \lesssim h^2|\log h|^2 \quad \text{for all } ~ 0 < h < h_*.
    \end{equation*}
\end{proposition}
\begin{proof}
    For simplicity, we only prove the estimate for the term $\|\bar{p}_{\alpha,1}- \hat p_{1}^{h}\|_{L^{1}(\Omega)}$. The case $k=2$ follows by similar arguments. Let $\tilde p^h\in\Vh$ be the unique solution to 
    \begin{equation}
        \label{def:tilde_varphi}
        (\varphi^h,\tilde p^h)_V = \sum_{i=1}^{n_1}\left(\bar{y}_{\alpha}(\bx_1^i) - y_1^i\right)\varphi^h(\bx_1^i)\quad\text{for all }\varphi^h\in\Vh.
    \end{equation}
    We note that $\tilde p^h$ corresponds to the FE approximation of $\bar{p}_{\alpha,1}$ (solution to \eqref{eq:adjoint_eq} with $k=1$) in $\Vh$. 
    An application of the triangle inequality and \cite[Lemma 5.3]{MR3973329} results in 
    \begin{align*}
        \|\bar{p}_{\alpha,1}- \hat p_{1}^{h}\|_{L^{1}(\Omega)}\leq\|\bar{p}_{\alpha,1}-\tilde p^h\|_{L^{1}(\Omega)} + \|\tilde p^h-\hat p_{1}^{h}\|_{L^{1}(\Omega)}\lesssim h^2|\log h|^{2}+\|\tilde p^h-\hat p_{1}^{h}\|_{L^{1}(\Omega)}.
    \end{align*}
    To control $\|\tilde p^h-\hat p_{1}^{h}\|_{L^{1}(\Omega)}$, we introduce the variable $\chi \in W_0^{1,\mathsf{p}}(\Omega)$ with $\mathsf{p}<\nicefrac{d}{d-1}$ as the unique solution to
    \begin{align*}
        (\varphi,\chi)_V= \sum_{i=1}^{n_1}\left(\bar{y}_{\alpha}(\bx_1^i) - \hat y^h(\bx_1^i)\right)\varphi(\bx_1^i) \quad \text{for all }\varphi\in W_0^{1,\mathsf{q}}(\Omega).
    \end{align*}
    We note that $\tilde p^h-\hat p_{1}^{h}$ is the unique FE approximation of $\chi$ in $\Vh$. 
    Therefore, an application of \cite[Lemma 5.3]{MR3973329}, the embedding $W_0^{1,\mathsf{p}}(\Omega)\hookrightarrow L^{1}(\Omega)$, and the stability estimate $\|\nabla \chi\|_{L^{\mathsf{p}}(\Omega;\R^{d})} \lesssim \sum_{i=1}^{n_1}|\bar{y}_{\alpha}(\bx_{1}^i) - \hat y^h(\bx_{1}^i)|$ (cf. \cite[Theorem 1]{MR812624}) yield
    \begin{align}
        \label{eq:estimate_tilde_hat_varphi}
        \begin{aligned}
            \|\tilde p^h-\hat p_{1}^{h}\|_{L^{1}(\Omega)}
            &\leq\|(\tilde p^h-\hat p_{1}^{h})-\chi\|_{L^{1}(\Omega)} + \|\chi\|_{L^{1}(\Omega)}\\
            &\lesssim h^2|\log h|^2 + \|\nabla \chi\|_{L^{\mathsf{p}}(\Omega;\R^{d})} \lesssim h^2|\log h|^2 + \sum_{i=1}^{n_1}|\bar{y}_{\alpha}(\bx_{1}^i) - \hat y^h(\bx_{1}^i)|.
        \end{aligned}
    \end{align}
    To estimate $\sum_{i=1}^{n_1}|\bar{y}_{\alpha}(\bx_{1}^i) - \hat y^h(\bx_{1}^i)|$ in \eqref{eq:estimate_tilde_hat_varphi}, we proceed as follows. 
    Since $\Omega_{1}$ is finite, there exist sets $\omega_1,\widehat{\omega}_{1}$ such that $\Omega_{1} \subset \omega_{1} \Subset \widehat{\omega}_{1} \Subset \Omega$ with $\widehat{\omega}_{1}$ smooth. 
    At the same time, since $\ua,\ub\in \R$, we have in particular $\bar{u}_{\alpha}\in L^{\infty}(\Omega)$. 
    These ingredients, in combination with the fact that $\hat y^h$ corresponds to the unique FE approximation of $\bar{y}_{\alpha}$ in $\Vh$, allow us to use \cite[Lemma 4.4 (i)]{MR3973329} to obtain
    \begin{align*}
        \sum_{i=1}^{n_1}|\bar{y}_{\alpha}(\bx_{1}^i) - \hat y^h(\bx_{1}^i)|
        \leq 
        n_1\|\bar{y}_{\alpha} - \hat y^h\|_{L^{\infty}(\omega_1)}
        \lesssim
        h^2|\log h|^{2}, \qquad ~ 0 < h < h_*.
    \end{align*}
    This concludes the proof.
\end{proof}

With these ingredients at hand, we present the following error estimate.

\begin{theorem}[Error estimate: WSM]\label{thm:estimate_WSM}
    Let $\bar u_\alpha $ and $\bar{u}_{\alpha}^{h}$ be the unique solutions to \eqref{eq:WSM_problem} and \eqref{eq:WSM_problem_discrete}, respectively. 
    Assume that $\ua,\ub$ are real constants.
    Then, there exists $h_*,h_{\star}>0$ such that    
    \begin{align*}
        \left(\sum_{k=1}^2 \alpha_k\lambda_k\right)^{\frac{1}{2}}\|\bar u_\alpha  - \bar{u}_{\alpha}^{h}\|_\U\lesssim h|\log h|\quad \text{for all } h < \min\{h_*,h_{\star}\}.
    \end{align*}
\end{theorem}
\begin{proof}
    Let $\pi_{0}:\U\to \Uh$ be the $\U$-orthogonal projection operator. We consider $u = \bar{u}_{\alpha}^{h}$ in \eqref{eq:var_ineq_WSM} and $u^h = \pi_{0}\bar u_\alpha $ in \eqref{eq:discrete_var_ineq_WSM}, and add the obtained inequalities. 
    This results in
    \begin{align}
        \label{eq:initial_estimate_I_II}
        \begin{aligned}
            \sum_{k=1}^2\alpha_k\lambda_k\|\bar u_\alpha-\bar{u}_{\alpha}^{h}\|_\U^2&\le\sum_{k=1}^2\alpha_k(\bar{p}_{\alpha,k}-\bar{p}_{\alpha,k}^{h}, \bar{u}_{\alpha}^{h} - \bar u_\alpha )_\U+\sum_{k=1}^2\alpha_k(\bar{p}_{\alpha,k}^{h} + \lambda_k \bar{u}_{\alpha}^{h}, \pi_{0}\bar u_\alpha  - \bar u_\alpha )_\U \\
            &\equalscolon\mathsf{I} + \mathsf{II}.
        \end{aligned}
    \end{align}
    Let us bound 
    \begin{align}
        \label{eq:estimate_I}
        \mathsf{I}=\alpha_1(\bar{p}_{\alpha,1}-\bar{p}_{\alpha,1}^{h}, \bar{u}_{\alpha}^{h} - \bar u_\alpha )_\U+\alpha_2(\bar{p}_{\alpha,2} - \bar{p}_{\alpha,2}^{h}, \bar{u}_{\alpha}^{h} - \bar u_\alpha )_\U\equalscolon\mathsf I_1+\mathsf I_2.
    \end{align}
    For clarity, we concentrate only on $\mathsf{I}_{1}$, since $\mathsf{I}_{2}$ can be treated analogously. 
    We first utilize the auxiliary variable $\hat p_{1}^{h}\in\Vh$, introduced in \eqref{def:hat_varphi}, and write
    \begin{align*}
        \mathsf I_1=  \alpha_1(\hat p_{1}^{h}-\bar{p}_{\alpha,1}^{h},\bar{u}_{\alpha}^{h}- \bar u_\alpha )_\U+ \alpha_1(\bar{p}_{\alpha,1}- \hat p_{1}^{h}, \bar{u}_{\alpha}^{h} - \bar u_\alpha )_\U.    
    \end{align*}
    Recall that $\hat p_{1}^{h}-\bar{p}_{\alpha,1}^{h}\in \Vh$ and $\bar{y}_{\alpha}^h-\hat y^h \in \Vh$ solve
    \begin{align*}
        & (\varphi^h,\hat p_{1}^{h} - \bar{p}_{\alpha,1}^{h})_V= \sum_{i=1}^{n_1}\left(\hat y^h(\bx_1^i)- \bar{y}_{\alpha}^h(\bx_1^i)\right)\varphi^h(\bx_1^i) &&\text{for all }\varphi^h\in\Vh,\\
        & (\bar{y}_{\alpha}^h-\hat y^h,\varphi^h)_V = (\bar{u}_{\alpha}^{h}- \bar u_\alpha ,\varphi^h)_{\U} && \text{for all }\varphi^h\in \Vh,
    \end{align*}
    respectively, with $\hat y^h\in\Vh$ defined in \eqref{def:hat_y}. Hence, choosing $\varphi^h =\bar{y}_{\alpha}^h-\hat y^h$ in the first equation and $\varphi^h= \hat p_{1}^{h}-\bar{p}_{\alpha,1}^{h}$ in the second equation, we infer that
    \begin{align}\label{eq:estimate_I1_<=0}
        \alpha_1(\hat p_{1}^{h}-\bar{p}_{\alpha,1}^{h}, \bar{u}_{\alpha}^{h}- \bar u_\alpha )_\U= -\alpha_{1} \sum_{i=1}^{n_1}\left(\hat y^h(\bx_1^i) - \bar{y}_{\alpha}^h(\bx_1^i)\right)^{2} \leq 0.
    \end{align}
    Consequently, $\mathsf I_1\le \alpha_1(\bar{p}_{\alpha,1}- \hat p_{1}^{h}, \bar{u}_{\alpha}^{h}-\bar u_\alpha)_\U$.
    To control the latter, we apply H\"older's inequality, Proposition \ref{prop:aux_estimate_hat_tilde}, and the fact that $\bar{u}_{\alpha},\bar{u}_{\alpha}^{h}\in L^{\infty}(\Omega)$:
    \begin{equation}\label{eq:estimate_I1}
        \mathsf{I}_{1}
        \leq 
        \alpha_1\|\bar{p}_{\alpha,1}- \hat p_{1}^{h}\|_{L^{1}(\Omega)}\|\bar{u}_{\alpha}^{h} - \bar u_\alpha \|_{L^{\infty}(\Omega)}
        \lesssim h^{2}|\log h|^{2}, \qquad ~ 0 < h < h_{*}.
    \end{equation}
    Therefore, in view of \eqref{eq:estimate_I1} and its analogous version for $\mathsf{I}_{2}$, \eqref{eq:estimate_I}, and \eqref{eq:initial_estimate_I_II}, we conclude that 
    \begin{align}
        \label{eq:estimate_sum_II}
    \left(\sum_{k=1}^2\alpha_k\lambda_k\right)\|\bar u_\alpha  - \bar{u}_{\alpha}^{h}\|_\U^2
    \leq C_1 h^{2}|\log h|^{2} + \mathsf{II}
    \end{align}
    for a constant $C_1>0$ and all $0 < h < h_{*}$. 
    To estimate $\mathsf{II}$ we write 
    \begin{align}
        \label{eq:estimate_II}
        \mathsf{II}=\alpha_1(\bar{p}_{\alpha,1}^{h} + \lambda_1 \bar{u}_{\alpha}^{h}, \pi_{0}\bar u_\alpha  - \bar u_\alpha )_\U + \alpha_2(\bar p_{\alpha,2}^{h} + \lambda_2 \bar{u}_{\alpha}^{h}, \pi_{0}\bar u_\alpha  - \bar u_\alpha )_\U\equalscolon\mathsf{II}_{1}+ \mathsf{II}_{2}. 
    \end{align}
    In what follows, we only estimate $\mathsf{II}_{1}$ since the estimation of $\mathsf{II}_{2}$ follows analogous steps.
    First, we note that
    \begin{align}
        \label{eq:estimate_II1}
        \begin{aligned}
            \mathsf{II}_{1}= &\,\alpha_1(\bar{p}_{\alpha,1} + \lambda_1 \bar u_\alpha , \pi_{0}\bar u_\alpha  - \bar u_\alpha )_\U + \alpha_1(\hat p_{1}^{h} - \bar{p}_{\alpha,1}, \pi_{0}\bar u_\alpha  - \bar u_\alpha )_\U\\
            & + \alpha_1(\bar{p}_{\alpha,1}^{h} - \hat p_{1}^{h}, \pi_{0}\bar u_\alpha  - \bar u_\alpha )_\U+ \alpha_1\lambda_{1}(\bar{u}_{\alpha}^{h} - \bar u_\alpha , \pi_{0}\bar u_\alpha  - \bar u_\alpha )_\U.
        \end{aligned}
    \end{align}
    To control the term $\alpha_1(\bar{p}_{\alpha,1} + \lambda_1 \bar u_\alpha , \pi_{0}\bar u_\alpha  - \bar u_\alpha )_\U$ in \eqref{eq:estimate_II1}, we use the orthogonality of $\pi_0$ to infer that
    \begin{align*}
        \alpha_1(\bar{p}_{\alpha,1} + \lambda_1 \bar u_\alpha, \pi_{0}\bar u_\alpha  - \bar u_\alpha )_\U
        = & ~
        \alpha_1((\bar{p}_{\alpha,1} + \lambda_1 \bar u_\alpha) - \pi_{0}(\bar{p}_{\alpha,1} + \lambda_1 \bar u_\alpha), \pi_{0}\bar u_\alpha  - \bar u_\alpha )_\U
        \\ 
        = & ~
        \alpha_1(\bar{p}_{\alpha,1} - \pi_{0}\bar{p}_{\alpha,1}, \pi_{0}\bar u_\alpha  - \bar u_\alpha )_\U - \lambda_{1}\|\pi_0\bar u_\alpha - \bar u_\alpha\|_\U^2 \\
        \leq & ~ \alpha_1(\bar{p}_{\alpha,1} - \pi_{0}\bar{p}_{\alpha,1}, \pi_{0}\bar u_\alpha  - \bar u_\alpha )_\U.
     \end{align*}
    Then, in view of the $W^{1,\infty}(\Omega)$-regularity of $\bar{u}_{\alpha}$, which stems from \eqref{eq:projection_control_WSM} and \cite[Lemma 6.1]{MR3973329}, and using that $\|\pi_{0}\bar{p}_{\alpha,1} - \bar{p}_{\alpha,1}\|_{L^{1}(\Omega)} \lesssim h\|\nabla \bar{p}_{\alpha,1}\|_{L^{1}(\Omega;\R^{d})}$ (cf. \cite[Proposition 1.135]{MR2050138}) we obtain
    \begin{align*}
        \alpha_1(\bar{p}_{\alpha,1} + \lambda_1 \bar u_\alpha , \pi_{0}\bar u_\alpha  - \bar u_\alpha )_\U
        \leq & ~
        \alpha_1\|\pi_{0}\bar{p}_{\alpha,1} - \bar{p}_{\alpha,1}\|_{L^{1}(\Omega)}\|\pi_{0}\bar u_\alpha  - \bar u_\alpha \|_{L^{\infty}(\Omega)} \\
        \lesssim & ~
        h^2\|\nabla \bar{p}_{\alpha,1}\|_{L^{1}(\Omega;\R^{d})}\|\nabla \bar{u}_{\alpha}\|_{L^{\infty}(\Omega;\R^{d})} 
        \lesssim
        h^2.
    \end{align*}
    The second term on the right-hand side of \eqref{eq:estimate_II1} is bounded in view of Proposition \ref{prop:aux_estimate_hat_tilde} and the $W^{1,\infty}(\Omega)$-regularity of $\bar{u}_{\alpha}$.
    These arguments yield
    \begin{align*}
        \alpha_1(\hat p_{1}^{h} - \bar{p}_{\alpha,1}, \pi_{0}\bar u_\alpha  - \bar u_\alpha )_{\U}
        \leq 
        \frac{\alpha_1}{2}\|\hat{p}_{1}^{h} - \bar{p}_{\alpha,1}\|_{L^{1}(\Omega)}^{2} + \frac{\alpha_1}{2}\|\pi_{0}\bar u_\alpha  - \bar u_\alpha \|_{L^{\infty}(\Omega)}^{2}
        \lesssim
        h^{4}|\log h|^{4} + h^{2} 
        \lesssim h^{2}.
    \end{align*}
    To estimate the third term on the right-hand side of \eqref{eq:estimate_II1}, we use H\"older's inequality and similar arguments to those that lead to \eqref{eq:estimate_tilde_hat_varphi}, to arrive at
    \begin{align}\nonumber
        \alpha_1(\bar{p}_{\alpha,1}^{h} - \hat p_{1}^{h}, \pi_{0}\bar u_\alpha  - \bar u_\alpha )_\U
        \leq & ~
        \alpha_1\|\bar{p}_{\alpha,1}^{h} - \hat p_{1}^{h}\|_{L^{1}(\Omega)}\|\pi_{0}\bar u_\alpha  - \bar u_\alpha \|_{L^{\infty}(\Omega)} \\
        \lesssim & ~
        \left(h^2|\log h|^2 + \sum_{i=1}^{n_1}|\bar{y}_{\alpha}^{h}(\bx_{1}^i) - \hat y^h(\bx_{1}^i)|\right)\|\pi_{0}\bar u_\alpha  - \bar u_\alpha \|_{L^{\infty}(\Omega)}.
        \label{eq:estimate_third_1}
    \end{align}
    We let $\phi\in \Y$ be the unique solution to $(\phi,\varphi)_{V} = (\bar{u}_{\alpha}^{h} - \bar{u}_{\alpha},\varphi)_{\U}$ for all $\varphi \in V$, and immediately note that $\bar{y}_{\alpha}^{h} - \hat y^h$ corresponds to the unique FE approximation of $\phi$ in $\Vh$.
    On the other hand, since the set $\Omega_{1}$ is finite, there exist sets $\omega_1,\widehat{\omega}_{1}$ such that $\Omega_{1} \subset \omega_{1} \Subset \widehat{\omega}_{1} \Subset \Omega$ with $\widehat{\omega}_{1}$ smooth. 
    Hence, the application of \cite[Lemma 4.4 (i)]{MR3973329} and the use of the stability estimate $\|\phi\|_{L^{\infty}(\omega_1)}\leq \|\phi\|_{L^{\infty}(\Omega)} \lesssim \|\bar{u}_{\alpha}^{h} - \bar{u}_{\alpha}\|_{\U}$ imply that
    \begin{align*}
        \sum_{i=1}^{n_1}|\bar{y}_{\alpha}(\bx_{1}^i) - \hat y^h(\bx_{1}^i)|
        \lesssim
        \|(\bar{y}_{\alpha} - \hat y^h) - \psi\|_{L^{\infty}(\omega_1)} + \|\psi\|_{L^{\infty}(\Omega)}
        \lesssim
        h^2|\log h|^{2} + \|\bar{u}_{\alpha}^{h} - \bar{u}_{\alpha}\|_{\U}, \qquad ~ 0 < h < h_{\star}.
    \end{align*}
    Using this estimate, the $W^{1,\infty}(\Omega)$-regularity of $\bar{u}_{\alpha}$, and Young's inequality in \eqref{eq:estimate_third_1}, we conclude that 
    \begin{align*}
    \alpha_1(\bar{p}_{\alpha,1}^{h} - \hat p_{1}^{h}, \pi_{0}\bar u_\alpha  - \bar u_\alpha )_\U
    \lesssim 
    h^3|\log h|^{2} + h\|\bar{u}_{\alpha}^{h} - \bar{u}_{\alpha}\|_{\U}
    \leq 
    C_{2}h^3|\log h|^{2} + C_{3}h^{2} + \frac{\alpha_{1}\lambda_{1}}{4}\|\bar{u}_{\alpha}^{h} - \bar u_\alpha \|_\U^{2}
    \end{align*}
    for constants $C_2,C_3 >0$.
    Finally, the term $\alpha_1\lambda_{1}(\bar{u}_{\alpha}^{h} - \bar u_\alpha , \pi_{0}\bar u_\alpha  - \bar u_\alpha )_\U$ in \eqref{eq:estimate_II1} is bounded using Cauchy-Schwarz inequality and estimate \cite[Proposition 1.135]{MR2050138} as follows
    \begin{align*}
      \alpha_1\lambda_{1}(\bar{u}_{\alpha}^{h} - \bar u_\alpha , \pi_{0}\bar u_\alpha  - \bar u_\alpha )_\U
      \leq & \,
          \alpha_{1}\lambda_{1}\|\pi_{0}\bar u_\alpha  - \bar u_\alpha \|_\U^{2} + \frac{\alpha_{1}\lambda_{1}}{4}\|\bar{u}_{\alpha}^{h} - \bar u_\alpha \|_\U^{2}\\
      \leq & \, C_4 h^{2} + \frac{\alpha_{1}\lambda_{1}}{4}\|\bar{u}_{\alpha}^{h} - \bar u_\alpha \|_\U^{2}
    \end{align*}
    for a constant $C_4>0$.
    Therefore, the combination of all the estimates derived for the terms on the right-hand side of \eqref{eq:estimate_II1} gives as a result that
    \begin{equation*}
        \mathsf{II}_{1}\leq C h^{2} + \frac{\alpha_{1}\lambda_{1}}{2}\|\bar{u}_{\alpha}^{h} - \bar u_\alpha \|_\U^{2} \quad (C > 0). 
    \end{equation*}
    Similarly, we have that $\mathsf{II}_{2} \leq C h^{2} + \nicefrac{\alpha_{2}\lambda_{2}}{2}\|\bar{u}_{\alpha}^{h} - \bar u_\alpha \|_\U^{2}$. We conclude the desired error estimate using the estimates obtained for $\mathsf{II}_{1}$ and $\mathsf{II}_{2}$ in \eqref{eq:estimate_II}, and using the resulting estimate in \eqref{eq:estimate_sum_II}.
\end{proof}

We now improve the error estimate of Theorem \ref{thm:estimate_WSM} in two dimensions. Its proof utilizes an FE estimate for the specific case $d=2$; cf., \cite{Scott73,Scott76} and \cite{MR812624}.

\begin{theorem}[Improved error estimate: WSM]
    \label{Theorem:ImprovedErrorWSM}
    Let $d = 2$. In the framework of Theorem~{\em\ref{thm:estimate_WSM}}, we have the following optimal error estimate:
        \begin{align*}
        \left(\sum_{k=1}^2 \alpha_k\lambda_k\right)^{\frac{1}{2}}\|\bar u_\alpha  - \bar{u}_{\alpha}^{h}\|_\U\lesssim h \quad \text{for all } 0 < h < h_{\star}.
    \end{align*}
\end{theorem}
\begin{proof}
    The proof relies on a more careful estimate for $\mathsf{I}$ in \eqref{eq:estimate_I}.
    To control this term, we concentrate, again, only on $\mathsf{I}_1$. 
    Writing $\mathsf I_1=  \alpha_1(\hat p_{1}^{h}-\bar{p}_{\alpha,1}^{h},\bar{u}_{\alpha}^{h}- \bar u_\alpha )_\U+ \alpha_1(\bar{p}_{\alpha,1}- \hat p_{1}^{h}, \bar{u}_{\alpha}^{h} - \bar u_\alpha )_\U$, using the estimate \eqref{eq:estimate_I1_<=0} and Cauchy-Schwarz inequality, we obtain 
    \begin{align*}
    \mathsf I_1
    \le \alpha_1(\bar{p}_{\alpha,1}- \hat p_{1}^{h}, \bar{u}_{\alpha}^{h}-\bar u_\alpha)_\U
    \leq \alpha_1\|\bar{p}_{\alpha,1}- \hat p_{1}^{h}\|_{\U}\|\bar{u}_{\alpha}^{h}-\bar u_\alpha\|_{\U}.
    \end{align*}
    On the other hand, the use of arguments similar to those that lead to the result of Proposition \ref{prop:aux_estimate_hat_tilde} in combination with the error estimate from \cite[Theorem 3]{MR812624} gives 
    \begin{align}\label{eq:analog_prop3.2}
     \|\bar{p}_{\alpha,1} - \hat p_{1}^{h}\|_{\U} \lesssim h,    
    \end{align}
    upon using that $d=2$.
    Hence, Young's inequality and the error estimate \eqref{eq:analog_prop3.2} allow us to conclude that
    \begin{equation*}
        \mathsf{I}_{1}
        \leq 
        \alpha_1\lambda_1\lambda_1^{-1}\|\bar{p}_{\alpha,1} - \hat p_{1}^{h}\|_{\U}\|\bar{u}_{\alpha}^{h} - \bar u_\alpha \|_{\U}
        \leq  Ch^{2} + \frac{\alpha_1\lambda_1}{4}\|\bar{u}_{\alpha}^{h} - \bar u_\alpha \|_{\U}^2, \quad (C>0).
    \end{equation*}
    Using, this estimate and its analogous version for $\mathsf{I}_{2}$, \eqref{eq:estimate_I}, and \eqref{eq:initial_estimate_I_II}, we obtain 
    \begin{align*}
    \left(\frac{3}{4}\sum_{k=1}^2\alpha_k\lambda_k\right)\|\bar u_\alpha  - \bar{u}_{\alpha}^{h}\|_\U^2
    \leq C h^{2} + \mathsf{II}.
    \end{align*}
    The term $\mathsf{II}$ is bounded as in the proof of Theorem \ref{thm:estimate_WSM}. 
    This concludes the proof.
\end{proof}


\subsection{Discretization of RPM}

Given $\zeta = (\zeta_1,\zeta_2) \in \mathcal{P}_{f}+\mathbb{R}^2_{\leq}$, we introduce the function $R_{\zeta}^{h}(u) \colonequals\nicefrac{1}{2}\left( (j_{1}^{h}(u) - \zeta_1)^2 + (j_{2}^{h}(u) - \zeta_2)^2\right)$ for every $u\in \U$.
The discrete version of the scalar-valued problem \eqref{eq:ref_point_prob} then reads
\begin{equation}\label{eq:ref_point_prob_discrete}
    \tag{$\mathbf{\hat P}_{\zeta}^{h}$}
    \min R_{\zeta}^{h}(u^h)\quad\text{s.t.}\quad u^h\in\Uadh.
\end{equation}

We have the following characterization for $(R_{\zeta}^{h})^{\prime}(u^h)$
\begin{align*}
    (R_{\zeta}^{h})^{\prime}(u^h)w^h =  \sum_{k=1}^{2}(j_{k}^{h}(u^h) - \zeta_{k})(p_{k}^{h} + \lambda_k u^h,w^h)_{\U} \quad \text{for all }w^h\in \mathbb{U}_{h}.
\end{align*}
With this characterization at hand, we present the following result, the proof of which follows as in Proposition \ref{prop:discrete_WSM}.

\begin{proposition}[Discrete RPM]
Let $\zeta =(\zeta_{1},\zeta_{2})\in \mathcal{P}_{f}+\mathbb{R}^2_{\leq}$.
Then, the scalar-valued optimal control problem \eqref{eq:ref_point_prob_discrete} has a unique solution $\bar{u}_{\zeta}^{h}\in \Uadh$.
Moreover, $\bar{u}_{\zeta}^{h}$ is an optimal solution to \eqref{eq:ref_point_prob_discrete} if and only if
\begin{equation}\label{eq:discrete_var_ineq_RPM}
 \sum_{k=1}^{2}(j_{k}^{h}(\bar{u}_{\zeta}^{h}) - \zeta_{k})(\bar{p}_{\zeta,k}^{h} + \lambda_k \bar{u}_{\zeta}^{h}, u^h - \bar{u}_{\zeta}^{h})_{\U} \geq 0 \quad \text{for all }u^h \in \Uadh,
\end{equation}
where $\bar{p}_{\zeta,k}^{h}$ denotes the unique solution to \eqref{eq:adjoint_eq_discrete} with $\bar{y}_{\zeta}^{h}=S_h\bar{u}_{\zeta}^{h}$.
\end{proposition}

Given $\zeta$ fixed, we use Algorithm \ref{alg:discrete_RPM} to solve problem \eqref{eq:ref_point_prob_discrete}.
This algorithm is also based on a Barzilai-Borwein gradient method.
\begin{algorithm}[ht]
	\caption{Solving problem \eqref{eq:ref_point_prob_discrete} with $\zeta$ fixed}
    \label{alg:discrete_RPM}
	\begin{algorithmic}[1]
        \REQUIRE{Problem data, $\zeta=(\zeta_1,\zeta_2)$ fixed, $tol=1$, and $u^{h,-1},u^{h,0}\in \Uadh$ with $u^{h,-1}\neq u^{h,0}$;}
        \STATE{Set $l=0$;}
        \WHILE{$tol > 10^{-8}$}
            \STATE{Obtain $y^{h,l}\in\Vh$ by solving \eqref{eq:discrete_Poisson} with $u=u^{h,l}$;}
            \STATE{Obtain $p_{1}^{h,l},p_{2}^{h,l}\in \Vh$ by solving \eqref{eq:adjoint_eq_discrete} with $y^h=y^{h,l}$;}
            \STATE{Compute step size
            \begin{align*}
                t_l = \frac{((R_{\zeta}^{h})^{\prime}(u^{h,l})-(R_{\zeta}^{h})^{\prime}(u^{h,{l-1}}),(R_{\zeta}^{h})^{\prime}(u^{h,l})- (R_{\zeta}^{h})^{\prime}(u^{h,{l-1}}))_\U}{((R_{\zeta}^{h})^{\prime}(u^{h,l})-(R_{\zeta}^{h})^{\prime}(u^{h,{l-1}}), u^{h,l} - u^{h,{l-1}})_\U};
            \end{align*}}
            \STATE{Set $u^{h,{l+1}}= \min\{\ub,\max\{\ua,u^{h,l}- \nicefrac{1}{t_l}(R_{\zeta}^{h})^{\prime}(u^{h,l})\}\}$ and \\$u_{\mathrm{ref}}= \min\{\ub,\max\{\ua,u^{h,l}- (R_{\zeta}^{h})^{\prime}(u^{h,l})\}\}$;} 
            \STATE{Set $tol = \|u^{h,{l+1}} - u_{\mathrm{ref}}\|_{\U}$ and $l=l+1$;}
        \ENDWHILE
        \RETURN{Optimal solution $\bar{u}_{\zeta}^{h}$ and its objective value $j^h(\bar{u}_{\zeta}^{h})$.}
	\end{algorithmic}
\end{algorithm}

In Algorithm \ref{alg:WSM} we show how we compute the discrete approximation of the set of Pareto points and the corresponding front.

\begin{algorithm}[ht]
    \label{alg:RPM}
    \caption{Reference point method}
    \begin{algorithmic}[1]
        \REQUIRE{Number $\ell_{\mathrm{max}}\in \mathbb{N}$ of stationary points, problem data, parameters $\eta^{\perp},\eta^{\|}>0$, and parameter $\varepsilon \ll 1$:}
        \STATE{Compute $\bar{u}_{\alpha}^{h,\mathrm{init}}$ solution to \eqref{eq:WSM_problem_discrete} with $\alpha=(1-\varepsilon,\varepsilon)$ (initial point);} 
        \STATE{Compute $\bar{u}_{\alpha}^{h,\mathrm{end}}$ solution to \eqref{eq:WSM_problem_discrete} with $\alpha=(\varepsilon,1-\varepsilon)$ (ending point);}
        \STATE{Set $\mathcal{P}_{s}^h = \{\bar{u}_{\alpha}^{h,\mathrm{init}}\}\cup \{\bar{u}_{\alpha}^{h,\mathrm{end}}\}, \mathcal{P}_f^h = \{j^{h}(\bar{u}_{\alpha}^{h,\mathrm{init}})\}\cup \{ j^h(\bar{u}_{\alpha}^{h,\mathrm{end}})\}$ and $\ell=1$;}
        \WHILE{$\zeta_1^{\ell} < j_{1}^{h}(\bar{u}_{\alpha}^{h,\mathrm{end}})$ and $\ell \leq \ell_{\mathrm{max}} - 1$}
            \STATE{Solve problem \eqref{eq:ref_point_prob_discrete} using Algorithm \ref{alg:discrete_RPM} with reference point $\zeta^{\ell}$, save the solution $\bar{u}_{\zeta}^{h,\ell}$ and its evaluation $j^h(\bar{u}_{\zeta}^{h,\ell})$;}
            \STATE{Set $\mathcal{P}_{s}^h=\mathcal{P}_{s}^h \cup \{ \bar{u}_{\zeta}^{h,\ell}\}$, $\mathcal{P}_{f}^h=\mathcal{P}_{f}^h \cup \{ j^h(\bar{u}_{\zeta}^{h,\ell})\}$, and $\ell=\ell+1$;}
        \ENDWHILE
        \RETURN{Discrete approximations $\mathcal{P}_{s}^h$ and $\mathcal{P}_f^h$ of Pareto stationary points and front.}
    \end{algorithmic}
\end{algorithm}

An important question to tackle when using the RPM is how to choose the reference points $\zeta$ in the numerical implementation.
We perform this following the approach of \cite[Section 3.2]{MR4491072}.
To make matter precise, let $\ell_{\mathrm{max}}$ be the maximal number of Pareto stationary points in the numerical implementation and let $u^{h,0}$ denote an initial starting point being a stationary point of the weighted-sum problem with weights $\alpha_{1} = 1 - \alpha_{2}$ and $\alpha_2 = \varepsilon \ll 1$. 
Then, the first reference point $\zeta^1$ (corresponding to the second point on the front) is chosen as
\begin{align}\label{def:zeta_1}
    \zeta^{1} = j^{h}(u^{h,0})^{\intercal} - (\mathfrak{h}^{\perp}, \mathfrak{h}^{\|}),
\end{align}
where $\mathfrak{h}^{\perp}, \mathfrak{h}^{\|}>0$ are suitable scaling parameters. 
Then, for $\ell = 1,\ldots, \ell_{\mathrm{max}}-2$ we take
\begin{align}\label{def:general_zeta}
    \zeta^{\ell+1} = j^{h}(u^{h,\ell})^{\intercal} + \mathfrak{h}^{\|}\frac{\eta^{\|}}{\|\eta^{\|}\|_{\mathbb{R}^{2}}} + \mathfrak{h}^{\perp}\frac{\eta^{\perp}}{\|\eta^{\perp}\|_{\mathbb{R}^{2}}},
\end{align}
with $\eta^{\perp} = \zeta^{\ell} - j^h(u^{h,\ell})$ and $\eta^{\|}=(-\eta_{2}^{\perp}, \eta_{1}^{\perp})$.
Note that due to the strong weighting of $j_{1}^{h}(u^{h,0})$, the Pareto front is approximately vertical in the area of the first
reference point, which motivates the initial choice $\eta^{\|}=(0,-1)$ and $\eta^{\perp}=(-1,0)$ in \eqref{def:zeta_1}. Using this update technique, we end up with the reference point method stated in Algorithm \ref{alg:RPM}.


\subsubsection{Error estimates for RPM}

Let $\zeta = (\zeta_1,\zeta_2) \in \mathcal{P}_{f}+\mathbb{R}^2_{\leq}$ be fixed.
In this section, we prove convergence rates for $\|\bar u_\zeta  - \bar{u}_{\zeta}^{h}\|_\U$, where $\bar u_\zeta $ and $\bar{u}_{\zeta}^{h}$ denote the unique solutions to \eqref{eq:ref_point_prob} and \eqref{eq:ref_point_prob_discrete}, respectively.

\begin{lemma}[Auxiliary estimate]\label{lemma:ju-juh}
    Let $v,w \in \Uad$.
    Assume that $\ua,\ub$ are real constants.
    Then,
    \begin{align*}
        \sum_{k=1}^{2}|j_{k}(v) - j_{k}^{h}(w)|
        \lesssim
        h^{2}|\log h|^{2} + \|v - w\|_{\U}  \quad \text{for all } h < h_*.
    \end{align*}
\end{lemma}
\begin{proof}
We begin by estimating the term $|j_{1}(v) - j_{1}^{h}(w)|$.
Let $y_v,y_w \in V$ be the unique solutions to 
\begin{align*}
    &(y_v,\varphi)_V= ( v,\varphi)_\U \quad \text{for all }\varphi\in V,\\
    &(y_w,\varphi)_V= ( w,\varphi)_\U \quad \text{for all }\varphi\in V,
\end{align*}
and let $y_{v}^{h}, y_{w}^{h} \in \mathbb{V}_{h}$ be their corresponding FE approximations, i.e., 
\begin{align*}
    &(y_{v}^{h},\varphi^h)_V= ( v,\varphi^h)_\U \quad \text{for all }\varphi^h\in \mathbb{V}_{h},\\
    &(y_{w}^{h},\varphi^h)_V= ( w,\varphi^h)_\U \quad \text{for all }\varphi^h\in \mathbb{V}_{h}.
\end{align*}
Then, we have that
\begin{align*}
|j_{1}(v) - j_{1}^{h}(w)|
\leq
\frac{1}{2}\sum_{i=1}^{n_1}|y_v(\bx_{1}^{i}) - y_{w}^{h}(\bx_{1}^{i})||y_v(\bx_{1}^{i}) + y_{w}^{h}(\bx_{1}^{i}) - 2y_{1}^{i}| + \frac{\lambda_{1}}{2}|(v-w,v+w)_{\U}|.
\end{align*}
Using the fact that $v,w\in \Uad$ and that (see \cite[Lemma 4.4 (i)]{MR3973329})
\begin{equation*}
|y_{w}^{h}(\bx_{1}^{i})| \leq |y_{w}^{h}(\bx_{1}^{i}) - y_{w}(\bx_{1}^{i})| + \|y_{w}\|_{L^{\infty}(\Omega)} \lesssim h^{2}|\log h|^{2} + \|w\|_{\U} \leq C, \quad i\in \{1,\ldots,n_1\},
\end{equation*}
we infer the estimate
\begin{align}\label{eq:j1v-jh1w}
|j_{1}(v) - j_{1}^{h}(w)|
\lesssim
\sum_{i=1}^{n_1}|y_v(\bx_{1}^{i}) - y_{w}^{h}(\bx_{1}^{i})| + \|v-w\|_{\U},
\end{align}
with a hidden constant that is independent of the discretization parameter. 
The term $\sum_{i=1}^{n_1}|y_v(\bx_{1}^{i}) - y_{w}^{h}(\bx_{1}^{i})|$ in \eqref{eq:j1v-jh1w} is controlled using $y_{w}$, the stability estimate $\|y_{v} - y_{w}\|_{L^{\infty}(\Omega)} \lesssim \|v-w\|_{\U}$, and \cite[Lemma 4.4 (i)]{MR3973329}:
\begin{equation*}
\sum_{i=1}^{n_1}|y_v(\bx_{1}^{i}) - y_{w}^{h}(\bx_{1}^{i})|
\leq 
n_1\|y_{v} - y_{w}\|_{L^{\infty}(\Omega)} + \sum_{i=1}^{n_1}|y_w(\bx_{1}^{i}) - y_{w}^{h}(\bx_{1}^{i})|
\lesssim
\|v-w\|_{\U} + h^{2}|\log h|^{2}.
\end{equation*}
Using this estimate in \eqref{eq:j1v-jh1w} we obtain $|j_{1}(v) - j_{1}^{h}(w)| \lesssim h^{2}|\log h|^{2} + \|v-w\|_{\U}$.
The estimation of $|j_{2}(v) - j_{2}^{h}(w)|$ follows analogous arguments. This concludes the proof.
\end{proof}

Before presenting the next result, we introduce $\tilde y^h \in \mathbb{V}_{h}$ as the unique solution to 
\begin{align}
    \label{def:tilde_y}
    (\tilde y^h,\varphi^h)_V= ( \bar u_\zeta,\varphi^h)_\U \quad \text{for all }\varphi^h\in \Vh,
\end{align}
and, for $k\in\{1,2\}$, the auxiliary variable $\tilde p_{k}^{h}\in \Vh$, solution to
\begin{align}
    \label{def:tilde_p}
    (\varphi^h,\tilde p_{k}^{h})_V=\sum_{i=1}^{n_k}\left(\tilde y^h(\bx_k^i) - y_k^i\right)\varphi^h(\bx_k^i) \quad \text{for all }\varphi^h \in \Vh.
\end{align}

\begin{theorem}[Error estimate: RPM]\label{thm:estimate_RPM}
    Let $\bar u_\zeta $ and $\bar{u}_{\zeta}^{h}$ be the unique solutions to \eqref{eq:ref_point_prob} and \eqref{eq:ref_point_prob_discrete}, respectively.
    Assume that $\ua,\ub\in \R$.
    Then, we have the error estimate
    \begin{align*}
      \sum_{k=1}^2 \lambda_k\left(\frac{3j_k(\bar{u}_{\zeta}) + j_{k}^{h}(\bar{u}_{\zeta}^{h})}{2} - 2\zeta_k\right)\|\bar u_\zeta  - \bar{u}_{\zeta}^{h}\|_\U^2\lesssim h^2|\log h|^2\quad \text{for all } h < h_*.
    \end{align*}
    In particular, if $h<h_*$ is small enough, then we have $\|\bar u_\zeta  - \bar{u}_{\zeta}^{h}\|_\U\lesssim h|\log h|$.
\end{theorem}
\begin{proof}
    Let $\pi_{0}:\U\to \Uh$ be the $\U$-orthogonal projection operator. 
    We consider $u = \bar{u}_{\zeta}^{h}$ in \eqref{eq:var_ineq_RPM} and $u^h = \pi_{0}\bar u_\zeta $ in \eqref{eq:discrete_var_ineq_RPM}, and add the obtained inequalities. 
    This results in
    \begin{align}\nonumber
           & \sum_{k=1}^{2}(j_{k}(\bar{u}_{\zeta}) - \zeta_{k})(\bar{p}_{\zeta,k} + \lambda_k \bar{u}_{\zeta}, \bar{u}_{\zeta}^{h} - \bar{u}_{\zeta})_\U +\sum_{k=1}^{2}(j_{k}^{h}(\bar{u}_{\zeta}^{h}) - \zeta_{k})(\bar{p}_{\zeta,k}^{h} + \lambda_k \bar{u}_{\zeta}^{h}, \bar u_\zeta - \bar{u}_{\zeta}^{h})_{\U} \\
            & ~ + \sum_{k=1}^{2}(j_{k}^{h}(\bar{u}_{\zeta}^{h}) - \zeta_{k})(\bar{p}_{\zeta,k}^{h} + \lambda_k \bar{u}_{\zeta}^{h}, \pi_{0}\bar u_\zeta - \bar{u}_{\zeta})_{\U} \geq 0.
            \label{eq:rpm_first_ineq}
    \end{align}
        The fact that the cost functions $j_k$ and $j_{k}^{h}$ ($k\in\{1,2\}$) are quadratic implies the following Taylor expansions:
        \begin{align*}
            & ~ j_{k}(\bar{u}_{\zeta}^{h}) =  j_{k}(\bar{u}_{\zeta}) + (\bar{p}_{\zeta,k} + \lambda_k \bar{u}_{\zeta}, \bar{u}_{\zeta}^{h} - \bar{u}_{\zeta})_\U + \frac{1}{2}\left(\sum_{i=1}^{n_k}(S(\bar{u}_{\zeta} - \bar{u}_{\zeta}^{h})(\bx_{k}^{i}))^{2} + \lambda_{k}\|\bar{u}_{\zeta} - \bar{u}_{\zeta}^{h}\|_{\U}^{2}\right),\\
            & ~ j_{k}^{h}(\bar{u}_{\zeta}) =  j_{k}^{h}(\bar{u}_{\zeta}^{h}) + (\bar{p}_{\zeta,k}^{h} + \lambda_k \bar{u}_{\zeta}^{h}, \bar u_\zeta - \bar{u}_{\zeta}^{h})_{\U} + \frac{1}{2}\left(\sum_{i=1}^{n_k}(S_{h}(\bar{u}_{\zeta} - \bar{u}_{\zeta}^{h})(\bx_{k}^{i}))^{2}  + \lambda_{k}\|\bar{u}_{\zeta} - \bar{u}_{\zeta}^{h}\|_{\U}^{2}\right).
        \end{align*}
        Using these expansions in \eqref{eq:rpm_first_ineq}, it follows that
        \begin{align*}
            &\sum_{k=1}^{2}(j_{k}(\bar{u}_{\zeta}) - \zeta_{k})\left(j_{k}(\bar{u}_{\zeta}^{h}) - j_{k}(\bar{u}_{\zeta}) - \tfrac{1}{2}\sum_{i=1}^{n_k}(S(\bar{u}_{\zeta} - \bar{u}_{\zeta}^{h})(\bx_{k}^{i}))^{2}-\tfrac{\lambda_{k}}{2}\|\bar{u}_{\zeta} - \bar{u}_{\zeta}^{h}\|_{\U}^{2}\right)  \\
            & + ~ \sum_{k=1}^{2}(j_{k}^{h}(\bar{u}_{\zeta}^{h}) - \zeta_{k})\left(j_{k}^{h}(\bar{u}_{\zeta}) - j_{k}^{h}(\bar{u}_{\zeta}^{h}) - \tfrac{1}{2}\sum_{i=1}^{n_k}(S_{h}(\bar{u}_{\zeta} - \bar{u}_{\zeta}^{h})(\bx_{k}^{i}))^{2}-\tfrac{\lambda_{k}}{2}\|\bar{u}_{\zeta} - \bar{u}_{\zeta}^{h}\|_{\U}^{2}\right)\\
            & + ~\sum_{k=1}^{2}(j_{k}^{h}(\bar{u}_{\zeta}^{h}) - \zeta_{k})(\bar{p}_{\zeta,k}^{h} + \lambda_k \bar{u}_{\zeta}^{h}, \pi_{0}\bar u_\zeta - \bar{u}_{\zeta})_{\U} \geq 0.
    \end{align*}
Hence, using the identity
\begin{align*}
    & \sum_{k=1}^{2}(j_{k}(\bar{u}_{\zeta}) - \zeta_{k})(j_{k}(\bar{u}_{\zeta}^{h}) - j_{k}(\bar{u}_{\zeta})) + \sum_{k=1}^{2}(j_{k}^{h}(\bar{u}_{\zeta}^{h}) - \zeta_{k})(j_{k}^{h}(\bar{u}_{\zeta}) - j_{k}^{h}(\bar{u}_{\zeta}^{h})) \\
    & ~ = 
     \sum_{k=1}^{2}(j_{k}(\bar{u}_{\zeta}) - \zeta_{k})(j_{k}(\bar{u}_{\zeta}^{h}) - j_{k}^{h}(\bar{u}_{\zeta}^{h}))
     + \sum_{k=1}^{2}(j_{k}(\bar{u}_{\zeta}) - \zeta_{k})(j_{k}^{h}(\bar{u}_{\zeta}^{h}) - j_{k}(\bar{u}_{\zeta})) \\
     & ~ + \sum_{k=1}^{2}(j_{k}^{h}(\bar{u}_{\zeta}^{h}) - \zeta_{k})(j_{k}^{h}(\bar{u}_{\zeta}) - j_{k}^{h}(\bar{u}_{\zeta}^{h})) \\
     & =  \sum_{k=1}^{2}(j_{k}(\bar{u}_{\zeta}) - \zeta_{k})(j_{k}(\bar{u}_{\zeta}^{h}) - j_{k}^{h}(\bar{u}_{\zeta}^{h})) - \sum_{k=1}^{2}(j_{k}(\bar{u}_{\zeta}) - j_{k}^{h}(\bar{u}_{\zeta}^{h}))^{2} \\
     & ~ + \sum_{k=1}^{2}(j_{k}^{h}(\bar{u}_{\zeta}^{h}) - \zeta_{k})(j_{k}^{h}(\bar{u}_{\zeta}) - j_{k}(\bar{u}_{\zeta})),
\end{align*}
we arrive at
\begin{align}\label{eq:initial_estimate_RPM}
    \begin{split}
            &\sum_{k=1}^{2}\frac{(j_{k}(\bar{u}_{\zeta}) - \zeta_{k})}{2}\left(\sum_{i=1}^{n_k}(S(\bar{u}_{\zeta} - \bar{u}_{\zeta}^{h})(\bx_{k}^{i}))^{2}+\lambda_{k}\|\bar{u}_{\zeta} - \bar{u}_{\zeta}^{h}\|_{\U}^{2}\right) + \sum_{k=1}^{2}(j_{k}(\bar{u}_{\zeta}) - j_{k}^{h}(\bar{u}_{\zeta}^{h}))^{2}\\
            & ~ + \sum_{k=1}^{2}\frac{(j_{k}^{h}(\bar{u}_{\zeta}^{h}) - \zeta_{k})}{2}\left(\sum_{i=1}^{n_k}(S_{h}(\bar{u}_{\zeta} - \bar{u}_{\zeta}^{h})(\bx_{k}^{i}))^{2}+\lambda_{k}\|\bar{u}_{\zeta} - \bar{u}_{\zeta}^{h}\|_{\U}^{2}\right)\\
             \leq ~  &  \sum_{k=1}^{2}(j_{k}(\bar{u}_{\zeta}) - \zeta_{k})(j_{k}(\bar{u}_{\zeta}^{h}) - j_{k}^{h}(\bar{u}_{\zeta}^{h})) + \sum_{k=1}^{2}(j_{k}^{h}(\bar{u}_{\zeta}^{h}) - \zeta_{k})(j_{k}^{h}(\bar{u}_{\zeta}) - j_{k}(\bar{u}_{\zeta})) \\
            & + ~\sum_{k=1}^{2}(j_{k}^{h}(\bar{u}_{\zeta}^{h}) - \zeta_{k})(\bar{p}_{\zeta,k}^{h} + \lambda_k \bar{u}_{\zeta}^{h}, \pi_{0}\bar u_\zeta - \bar{u}_{\zeta})_{\U}.
    \end{split}
\end{align}
Thus, using in \eqref{eq:initial_estimate_RPM} the identity 
\begin{align*}
    j_{k}(\bar{u}_{\zeta}^{h}) - j_{k}^{h}(\bar{u}_{\zeta}^{h})
    = ~ & 
    j_{k}(\bar{u}_{\zeta}) + (\bar{p}_{\zeta,k} + \lambda_{k}\bar{u}_{\zeta}, \bar{u}_{\zeta}^{h} - \bar{u}_{\zeta})_{\U} + \frac{1}{2}\sum_{i=1}^{n_k}(S(\bar{u}_{\zeta} - \bar{u}_{\zeta}^{h})(\bx_k^i))^{2} + \frac{\lambda_{k}}{2}\|\bar{u}_{\zeta} - \bar{u}_{\zeta}^{h}\|_{\U}^{2} \\
    & - j_{k}^{h}(\bar{u}_{\zeta}) - (\tilde{p}_{k}^{h} + \lambda_{k}\bar{u}_{\zeta}^{h}, \bar{u}_{\zeta}^{h} - \bar{u}_{\zeta})_{\U} - \frac{1}{2}\sum_{i=1}^{n_k}(S_h(\bar{u}_{\zeta} - \bar{u}_{\zeta}^{h})(\bx_k^i))^{2} - \frac{\lambda_{k}}{2}\|\bar{u}_{\zeta} - \bar{u}_{\zeta}^{h}\|_{\U}^{2}\\
= &  ~  (j_{k}(\bar{u}_{\zeta}) -  j_{k}^{h}(\bar{u}_{\zeta})) + (\bar{p}_{\zeta,k} - \tilde{p}_{k}^{h}, \bar{u}_{\zeta}^{h} - \bar{u}_{\zeta})_{\U} - \lambda_{k}\|\bar{u}_{\zeta} - \bar{u}_{\zeta}^{h}\|_{\U}^2\\
 & ~ + \frac{1}{2}\sum_{i=1}^{n_k}(S(\bar{u}_{\zeta} - \bar{u}_{\zeta}^{h})(\bx_k^i))^{2} - \frac{1}{2}\sum_{i=1}^{n_k}(S_h(\bar{u}_{\zeta} - \bar{u}_{\zeta}^{h})(\bx_k^i))^{2},
\end{align*}
which follows from a Taylor expansion, we conclude that 
    \begin{align}\label{eq:estimate_I_II_III}
    \begin{split}
            &\left(\sum_{k=1}^{2}\left(\frac{3j_{k}(\bar{u}_{\zeta}) + j_{k}^{h}(\bar{u}_{\zeta}^{h})}{2} - 2\zeta_{k}\right)\lambda_{k}\right)\|\bar{u}_{\zeta} - \bar{u}_{\zeta}^{h}\|_{\U}^{2} + \sum_{k=1}^{2}(j_{k}(\bar{u}_{\zeta}) - j_{k}^{h}(\bar{u}_{\zeta}^{h}))^{2}\\
            & ~ + \sum_{k=1}^{2}\left(\frac{j_{k}(\bar{u}_{\zeta}) + j_{k}^{h}(\bar{u}_{\zeta}^{h})}{2} - \zeta_{k}\right)\left(\sum_{i=1}^{n_k}(S_{h}(\bar{u}_{\zeta} - \bar{u}_{\zeta}^{h})(\bx_{k}^{i}))^{2}\right)\\
             \leq ~  &  \sum_{k=1}^{2}(j_{k}(\bar{u}_{\zeta}) - j_{k}^{h}(\bar{u}_{\zeta}^{h}))(j_{k}(\bar{u}_{\zeta}) -  j_{k}^{h}(\bar{u}_{\zeta})) + \sum_{k=1}^{2}(j_{k}(\bar{u}_{\zeta}) - \zeta_{k})(\bar{p}_{\zeta,k} - \tilde{p}_{k}^{h}, \bar{u}_{\zeta}^{h} - \bar{u}_{\zeta})_{\U}\\
             &  + \sum_{k=1}^{2}(j_{k}^{h}(\bar{u}_{\zeta}^{h}) - \zeta_{k})(\bar{p}_{\zeta,k}^{h} + \lambda_k \bar{u}_{\zeta}^{h}, \pi_{0}\bar u_\zeta - \bar{u}_{\zeta})_{\U} =: \mathsf{I} + \mathsf{II} + \mathsf{III}.
    \end{split}
    \end{align}
We recall that the auxiliary term $\tilde{p}_{k}^{h}$ ($k\in\{1,2\}$) is defined in \eqref{def:tilde_p}.
We immediately note that, for $h>0$ sufficiently small, the inequality $j_{k}^{h}(\bar{u}_{\zeta}^{h}) - \zeta_{k}>0$ ($k\in\{1,2\}$) holds; this follows from the fact that $\|\bar{u}_{\zeta} - \bar{u}_{\zeta}^{h}\|_{\U} \to 0$ when $h \to 0$ in combination with Lemma \ref{lemma:ju-juh}.
Consequently, all the terms on the left-hand side of \eqref{eq:estimate_I_II_III} are nonnegative when $h>0$ is small enough.

We now estimate $\mathsf{I}$, $\mathsf{II}$, and $\mathsf{III}$. 
To estimate $\mathsf{I}$ we use Lemma \ref{lemma:ju-juh} and Young's inequality to arrive at 
\begin{equation*}
\mathsf{I}
\lesssim
 (h^{2}|\log h|^{2} + \|\bar{u}_{\zeta} - \bar{u}_{\zeta}^{h}\|_{\U})h^{2}|\log h|^{2}
 \lesssim
 (1 + \varepsilon^{-1}) h^{4}|\log h|^{4} + \varepsilon\|\bar{u}_{\zeta} - \bar{u}_{\zeta}^{h}\|_{\U}^{2},
\end{equation*}
with $\varepsilon > 0$ arbitrary.
To control the term $\mathsf{II}$, we use H\"older's inequality, an analogous version of Proposition \ref{prop:aux_estimate_hat_tilde} for $\|\bar{p}_{\zeta,k} - \tilde{p}_{k}^{h}\|_{L^{1}(\Omega)}$, and the fact that $\bar{u}_{\zeta}^{h},\bar{u}_{\zeta}\in \Uad$ to obtain
\begin{equation*}
\mathsf{II}
\lesssim
 \|\bar{p}_{\zeta,k} - \tilde{p}_{k}^{h}\|_{L^{1}(\Omega)}\|\bar{u}_{\zeta}^{h} - \bar{u}_{\zeta}\|_{L^{\infty}(\Omega)}
 \lesssim
  h^{2}|\log h|^2.
\end{equation*}
The term $\mathsf{III}$ can be bounded as the term $\mathsf{II}$ in the proof of Theorem \ref{thm:estimate_WSM} (cf. \eqref{eq:estimate_II}). 
Thus, we obtain that 
\begin{equation}\label{eq:final_III_RPM}
\mathsf{III}
\lesssim
 \varepsilon^{-1} h^2 + \varepsilon \|\bar{u}_{\zeta}^{h} - \bar{u}_{\zeta}\|_{\U}^{2},
\end{equation}
with $\varepsilon > 0$ arbitrary, upon using that $\sum_{k=1}^{2}|j_{k}^{h}(\bar{u}_{\zeta}^{h}) - \zeta_{k}|\leq C$, for $C>0$.

We conclude the desired bound by replacing the estimates obtained for $\mathsf{I}$, $\mathsf{II}$, and $\mathsf{III}$ in \eqref{eq:estimate_I_II_III} and taking $\varepsilon > 0$ small enough.
\end{proof}

As in the discrete approximation of the weighted-sum method (see Theorem~\ref{Theorem:ImprovedErrorWSM}), we can improve the error estimate of Theorem \ref{thm:estimate_RPM} in two dimensions.

\begin{theorem}[Improved error estimate: RPM]\label{thm:improved_RPM}
    Let $d = 2$. In the framework of Theorem~{\em\ref{thm:estimate_RPM}}, we have the following optimal error estimate:
        \begin{align*}
        \|\bar u_\zeta  - \bar{u}_{\zeta}^{h}\|_\U\lesssim h \quad \text{for all } 0 < h < h_{*}.
    \end{align*}
\end{theorem}
\begin{proof}
     The proof relies on different estimates for $\mathsf{I}$ and $\mathsf{II}$ in \eqref{eq:estimate_I_II_III}.
     First, we note that similar arguments to the ones that lead to Lemma \ref{lemma:ju-juh} in combination with the error estimate from \cite[Theorem 3]{MR812624} give
    \begin{align}\label{eq:analog_lemma3.6}
    \sum_{k=1}^{2}|j_{k}(v) - j_{k}^{h}(w)|
    \lesssim
    h + \|v - w\|_{\U}. 
    \end{align}
    Then, we bound $\mathsf{I}$ using the estimate \eqref{eq:analog_lemma3.6} as follows:
    \begin{align*}
       \mathsf{I}
       \lesssim
       (h + \|\bar{u}_{\zeta} - \bar{u}_{\zeta}^{h}\|_{\U}^2)h 
       \lesssim
       (1 + \varepsilon^{-1})h^2 + \varepsilon\|\bar{u}_{\zeta} - \bar{u}_{\zeta}^{h}\|_{\U}^2,
       \end{align*}
    with $\varepsilon > 0$ arbitrary.
    For the term $\mathsf{II}$ in \eqref{eq:estimate_I_II_III}, we use the error bound $\|\bar{p}_{\zeta,k} - \tilde{p}_{k}^{h}\|_{\U} \lesssim h$ and obtain
\begin{align*}
     \mathsf{II}
     \lesssim
     \|\bar{p}_{\zeta,k} - \tilde{p}_{k}^{h}\|_{\U}\|\bar{u}_{\zeta}^{h} - \bar{u}_{\zeta}\|_{\U}
    \lesssim
       \varepsilon^{-1} h^{2} + \varepsilon\|\bar{u}_{\zeta} - \bar{u}_{\zeta}^{h}\|_{\U}^2,
\end{align*}
    with $\varepsilon > 0$ arbitrary.
    Therefore, the use of the previous two estimates and \eqref{eq:final_III_RPM} in \eqref{eq:estimate_I_II_III}, and taking $\varepsilon > 0$ and $h > 0$ sufficiently small, yields the desired result.        
\end{proof}


\section{Numerical examples}\label{sec:num_exp}

In this section, we present some numerical experiments that show the performance of the scalarization methods presented in Section \ref{sec:scala_tech} when approximating Pareto stationary points and front.
The experiments were carried out with a code implemented in \texttt{MATLAB$^\copyright$ (R2024a)}.

In both numerical examples, we consider the domains $\Omega=(0,1)^2$, $\Omega_{1}=\{(0.75,0.25)\}$, $\Omega_{2}=\{(0.25,0.75)\}$, the desired states $y_{1}^{1} = 6$, $y_{2}^{1}=-2$, and the bilateral bounds $\ua=-7$, $\ub=15$. For the regularization parameters $\lambda_1$ and $\lambda_2$ in \eqref{def:cost_funct}, we consider four different configurations: $(\lambda_1,\lambda_2)=(1,1)$, $\lambda_1=1$ and $\lambda_2=0.1$, $\lambda_1=0.1$ and $\lambda_2=1$, and $\lambda_1=\lambda_2=0.1$. 
In both numerical examples, we consider the domains $\Omega=(0,1)^2$, $\Omega_{1}=\{(0.75,0.25)\}$, $\Omega_{2}=\{(0.25,0.75)\}$, the desired states $y_{1}^{1} = 6$, $y_{2}^{1}=-2$, and the bilateral bounds $\ua=-7$, $\ub=15$. For the regularization parameters $\lambda_1$ and $\lambda_2$ in \eqref{def:cost_funct}, we consider four different configurations: $(\lambda_1,\lambda_2)=(1,1)$, $\lambda_1=1$ and $\lambda_2=0.1$, $\lambda_1=0.1$ and $\lambda_2=1$, and $\lambda_1=\lambda_2=0.1$. 
In both numerical examples, we consider the domains $\Omega=(0,1)^2$, $\Omega_{1}=\{(0.75,0.25)\}$, $\Omega_{2}=\{(0.25,0.75)\}$, the desired states $y_{1}^{1} = 6$, $y_{2}^{1}=-2$, and the bilateral bounds $\ua=-7$, $\ub=15$. For the regularization parameters $\lambda_1$ and $\lambda_2$ in \eqref{def:cost_funct}, we study four different configurations: $(\lambda_1,\lambda_2)=(1,1)$, $(\lambda_1,\lambda_2)=(1,0.1)$, $(\lambda_1,\lambda_2)=(0.1,1)$ and $(\lambda_1,\lambda_2)=(0.1,0.1)$. 

We finally mention that, in the absence of an exact solution, we compute the error committed in the approximation by taking as a reference solution $\bar{u}_{\alpha}$ (resp. $\bar{u}_{\zeta}$) the discrete optimal control $\bar{u}_{\alpha}^{h}$  (resp. $\bar{u}_{\zeta}^{h}$) obtained on a finer triangulation $\mathcal{T}_{h}$: the mesh $\mathcal{T}_{h}$ is such that $h = 2^{-8}$. 

\subsection{Example 1 (WSM).} 

We present the results obtained for this example in Figures \ref{fig:ex1_1}-\ref{fig:ex1_3}, and Table \ref{table:ex1}.
In Fig. \ref{fig:ex1_1}, the Pareto fronts and the approximation error $\|\mathcal{J}(\bar{y}_{\alpha}^{h},\bar{u}_{\alpha}^{h}) - \mathcal{J}(\bar{y}_{\alpha},\bar{u}_{\alpha})\|_{2}$ are shown with $50$ different values of $\alpha$ considering mesh refinement using the weighted-sum method, and for the four different configurations of $(\lambda_1,\lambda_2)$ mentioned above. For these four configurations, we observe that the approximate Pareto front seems to converge to the continuous one as the discretization parameter $h$ decreases. 
We also observe that, even for the same problem data, the shape of the curves that represent the Pareto fronts changes as the values of $\lambda_1$ and $\lambda_2$ change.
In Fig. \ref{fig:ex1_3}, we show the approximate optimal control $\bar{u}_{\alpha}^{h}$ obtained with $\lambda_1=\lambda_2=0.1$, $h=2^{-6}$, and for four different instances of $\alpha$, namely $\alpha = (0.2,0.8)$, $\alpha = (0.4,0.6)$, $\alpha = (0.6,0.4)$, and $\alpha = (0.8,0.2)$.
Finally, in Table \ref{table:ex1}, we present the approximation error $\|\bar{u}_{\alpha} - \bar{u}_{\alpha}^{h}\|_{\U}$ for different cases of $\alpha$.
We observe that optimal experimental rates of convergence $\mathcal{O}(h)$ are attained for the four different choices of $\alpha$, which is in agreement with Theorem \ref{Theorem:ImprovedErrorWSM}.

\begin{figure}
\begin{center}
\begin{minipage}[c]{0.38\textwidth}\centering
\includegraphics[trim={0 0 0 0},clip,width=6.0cm,height=4.0cm,scale=0.1]{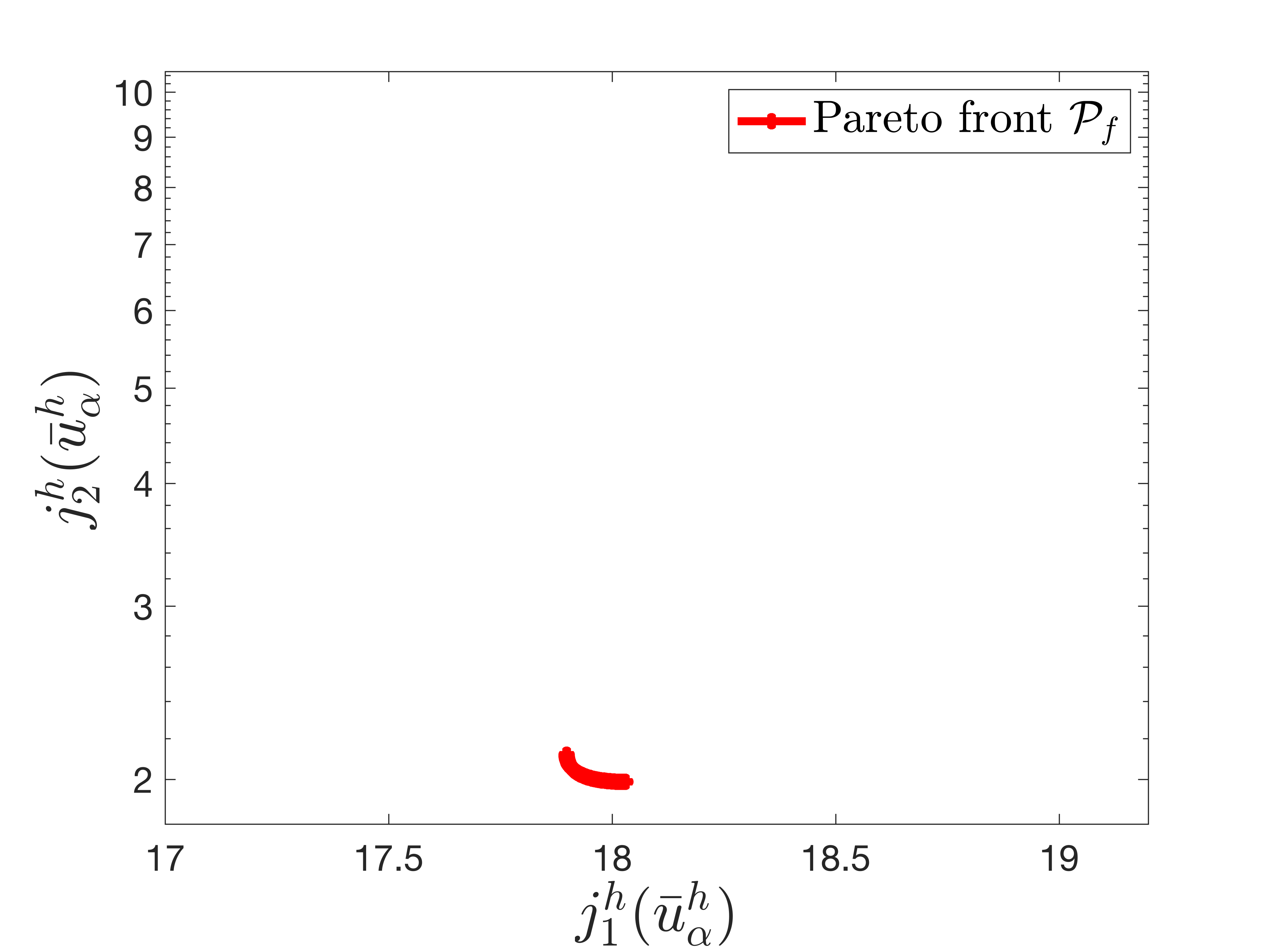}
\qquad
{\small{(1.A)}}
\end{minipage}
\begin{minipage}[c]{0.38\textwidth}\centering
\includegraphics[trim={0 0 0 0},clip,width=5.5cm,height=4.0cm,scale=0.1]{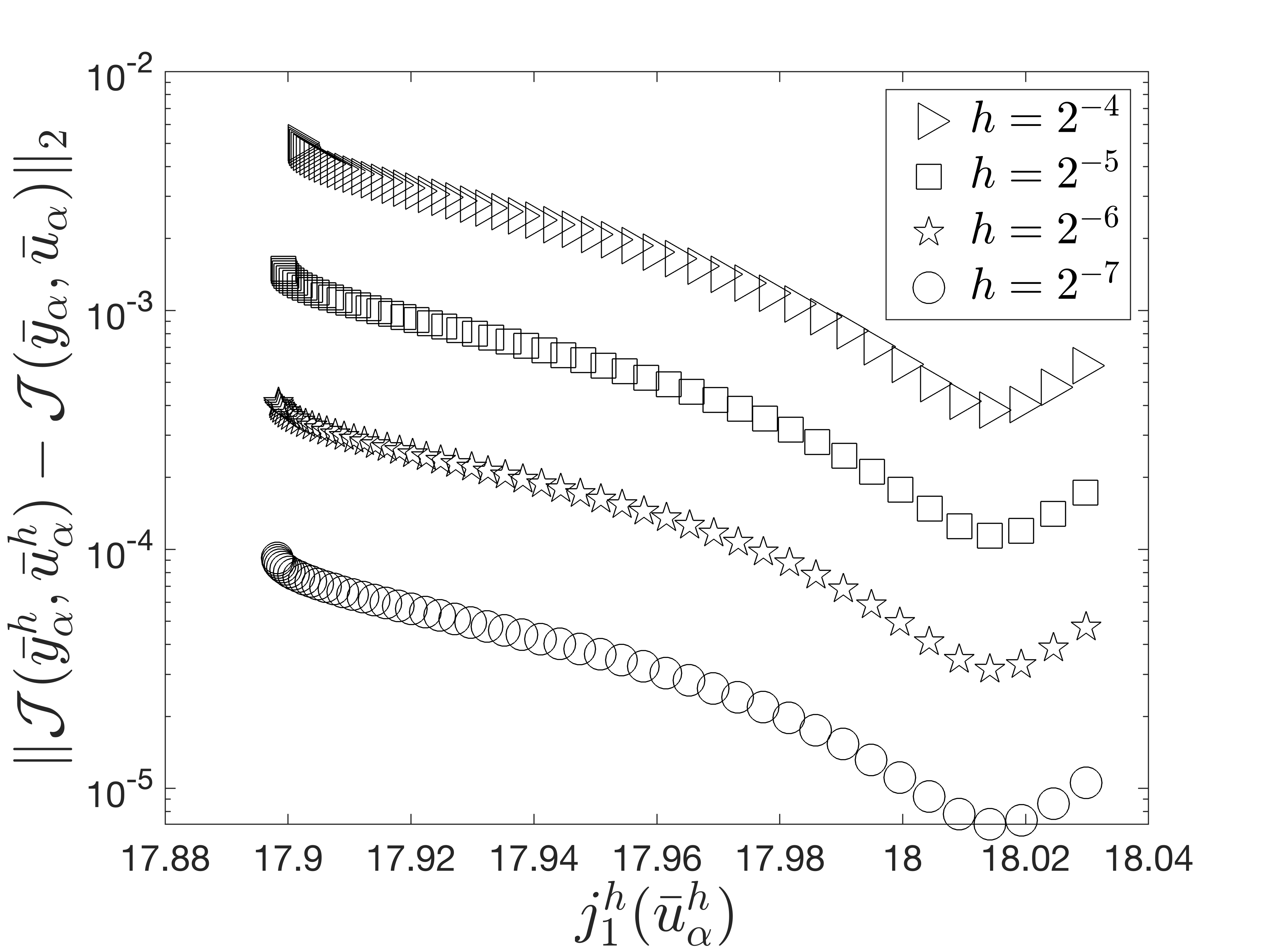}
\qquad
{\small{(1.B)}}
\end{minipage}
\begin{minipage}[c]{0.38\textwidth}\centering
\includegraphics[trim={0 0 0 0},clip,width=5.5cm,height=4.0cm,scale=0.1]{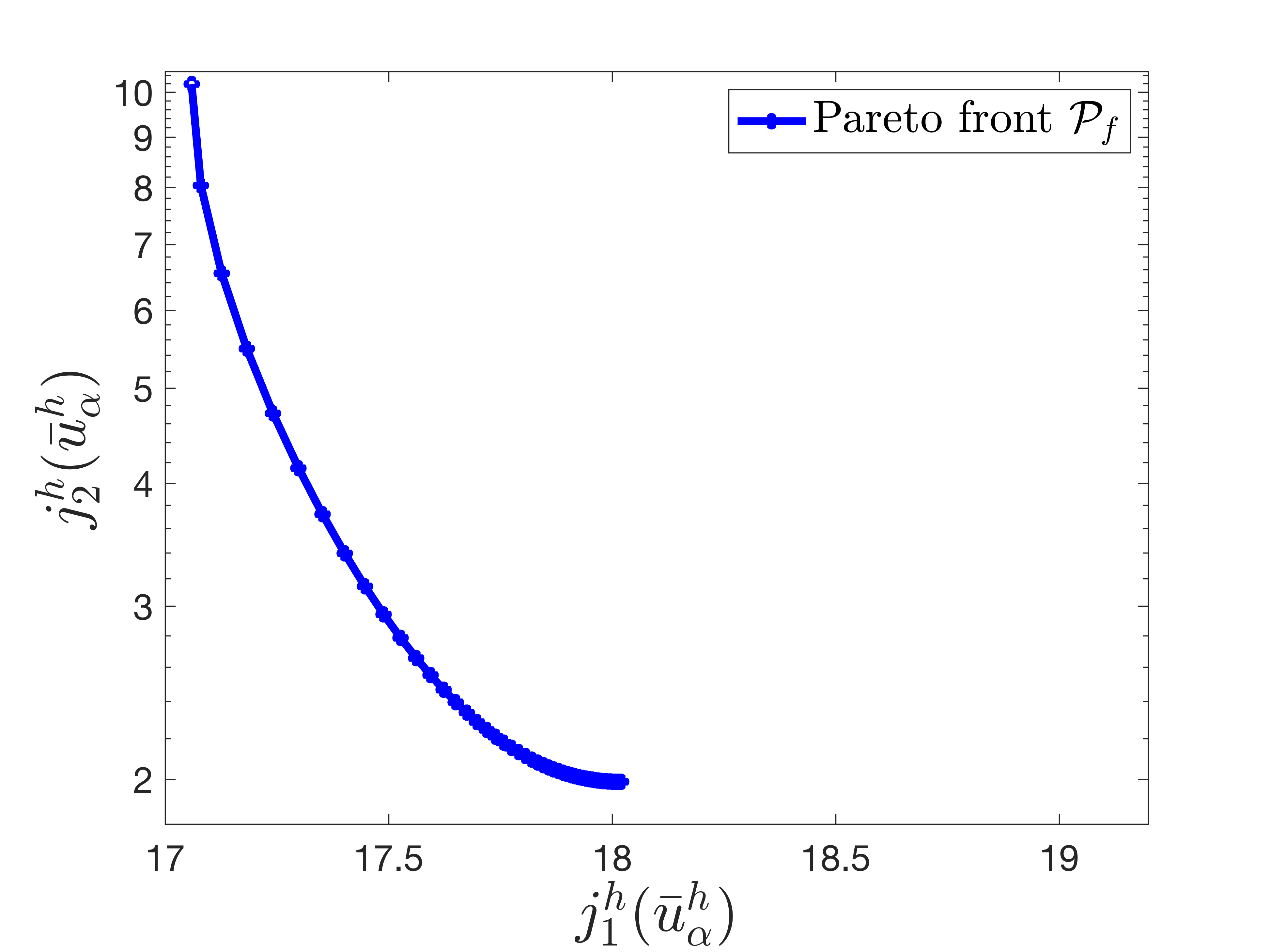}
\qquad
{\small{(1.C)}}
\end{minipage}
\begin{minipage}[c]{0.38\textwidth}\centering
\includegraphics[trim={0 0 0 0},clip,width=5.5cm,height=4.0cm,scale=0.1]{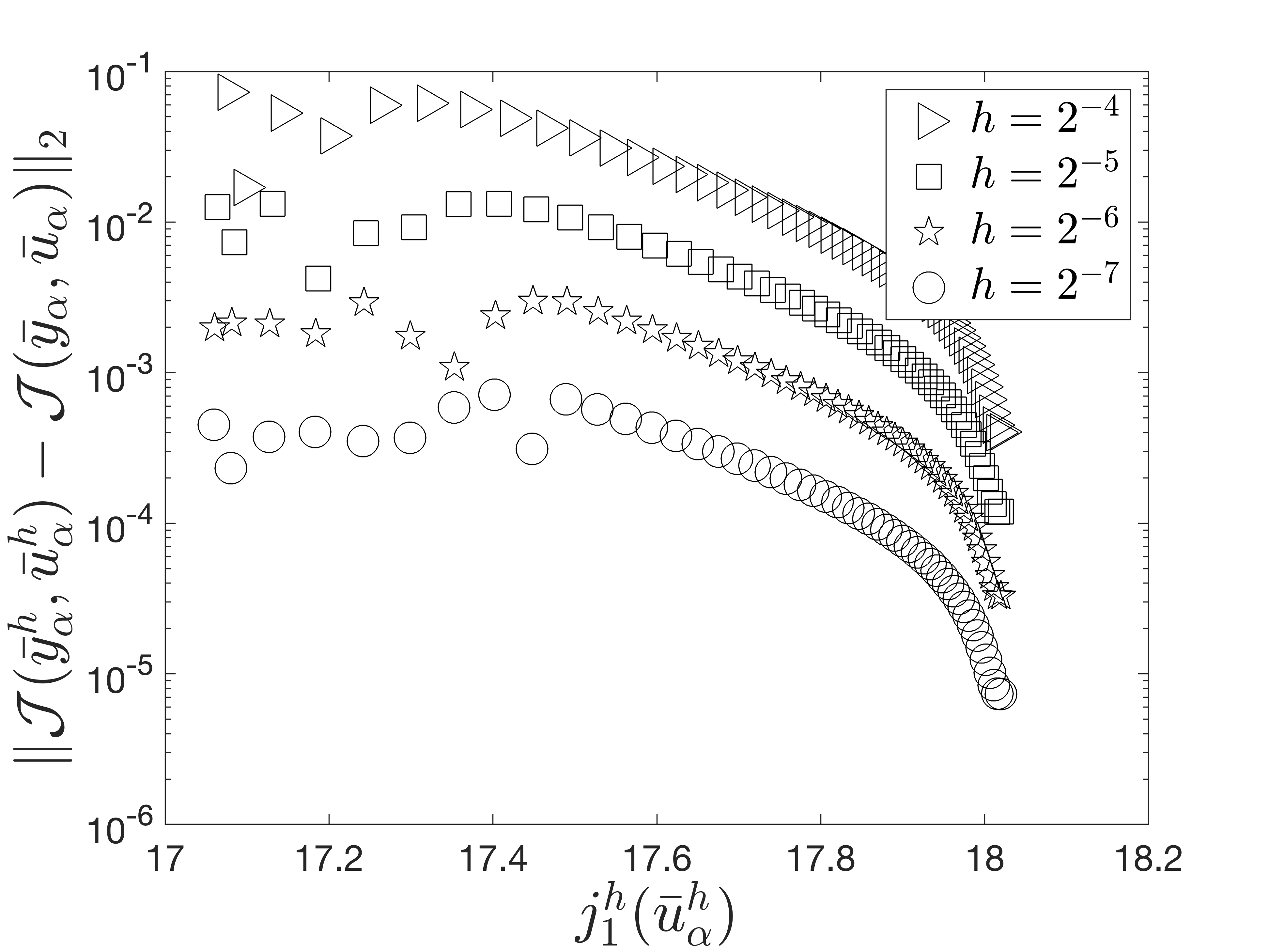}
\qquad
{\small{(1.D)}}
\end{minipage}
\begin{minipage}[c]{0.38\textwidth}\centering
\includegraphics[trim={0 0 0 0},clip,width=5.5cm,height=4.0cm,scale=0.1]{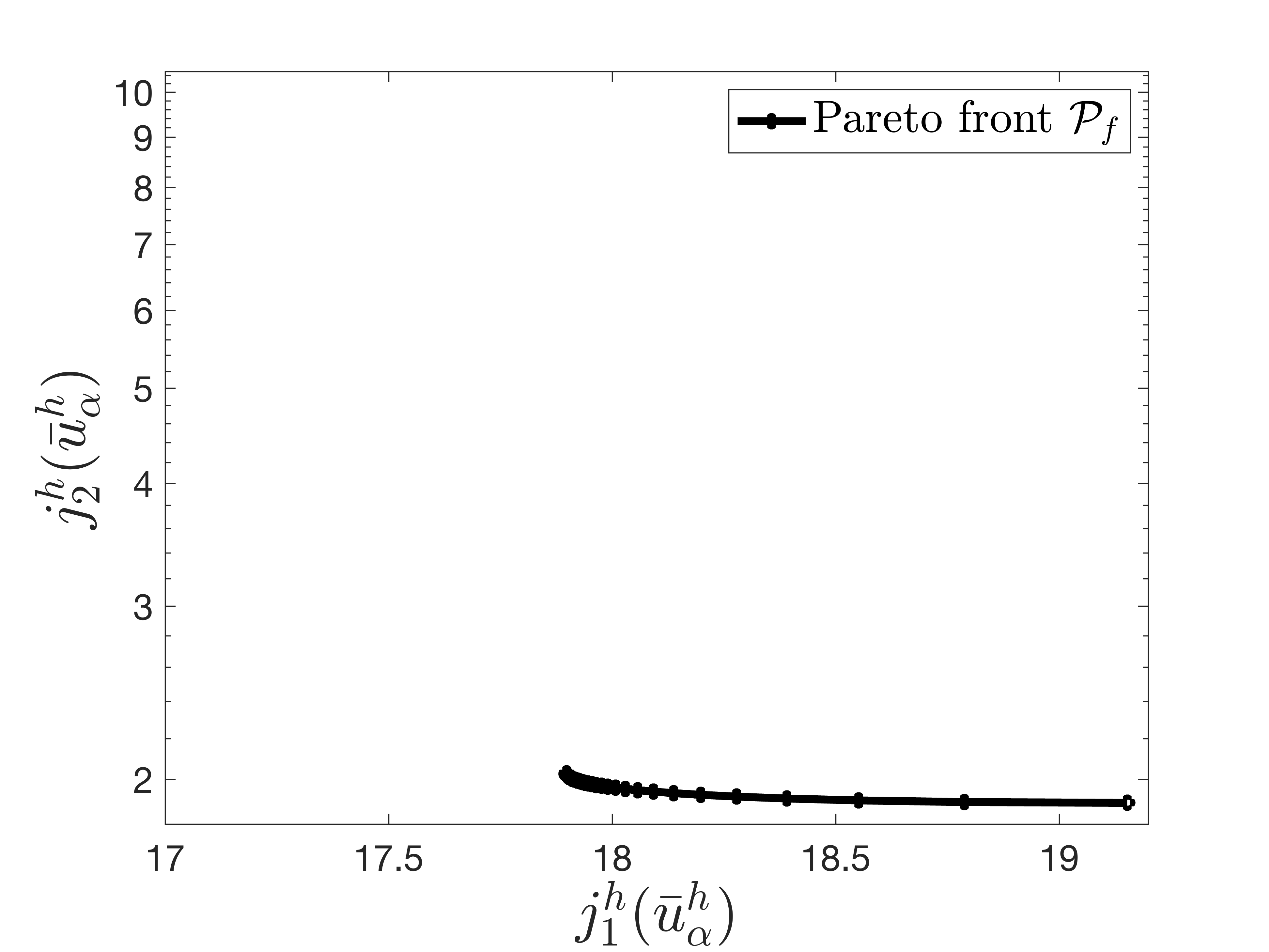}
\qquad
{\small{(1.E)}}
\end{minipage}
\begin{minipage}[c]{0.38\textwidth}\centering
\includegraphics[trim={0 0 0 0},clip,width=5.5cm,height=4.0cm,scale=0.1]{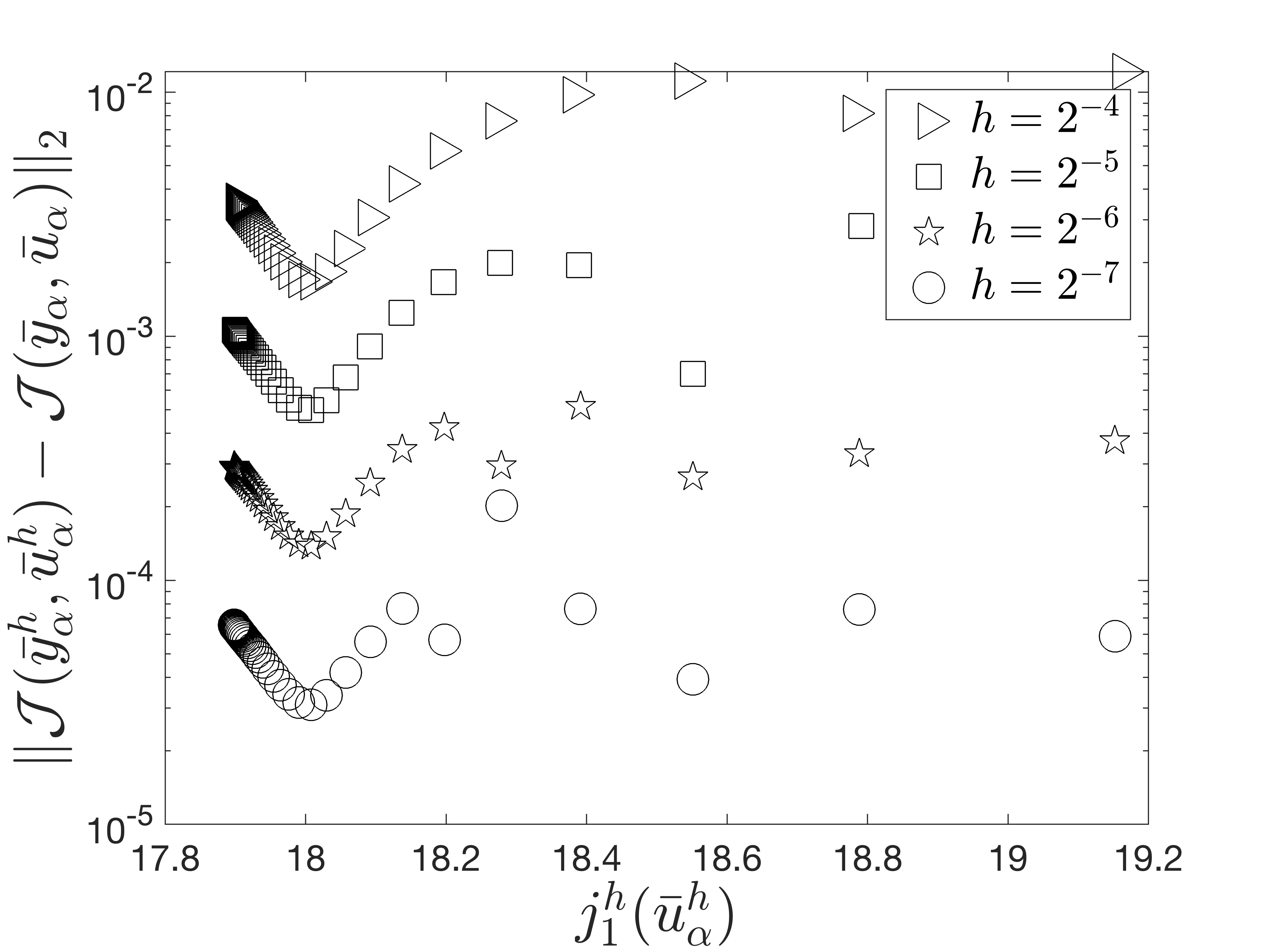}
\qquad
{\small{(1.F)}}
\end{minipage}
\begin{minipage}[c]{0.38\textwidth}\centering
\includegraphics[trim={0 0 0 0},clip,width=5.5cm,height=4.0cm,scale=0.1]{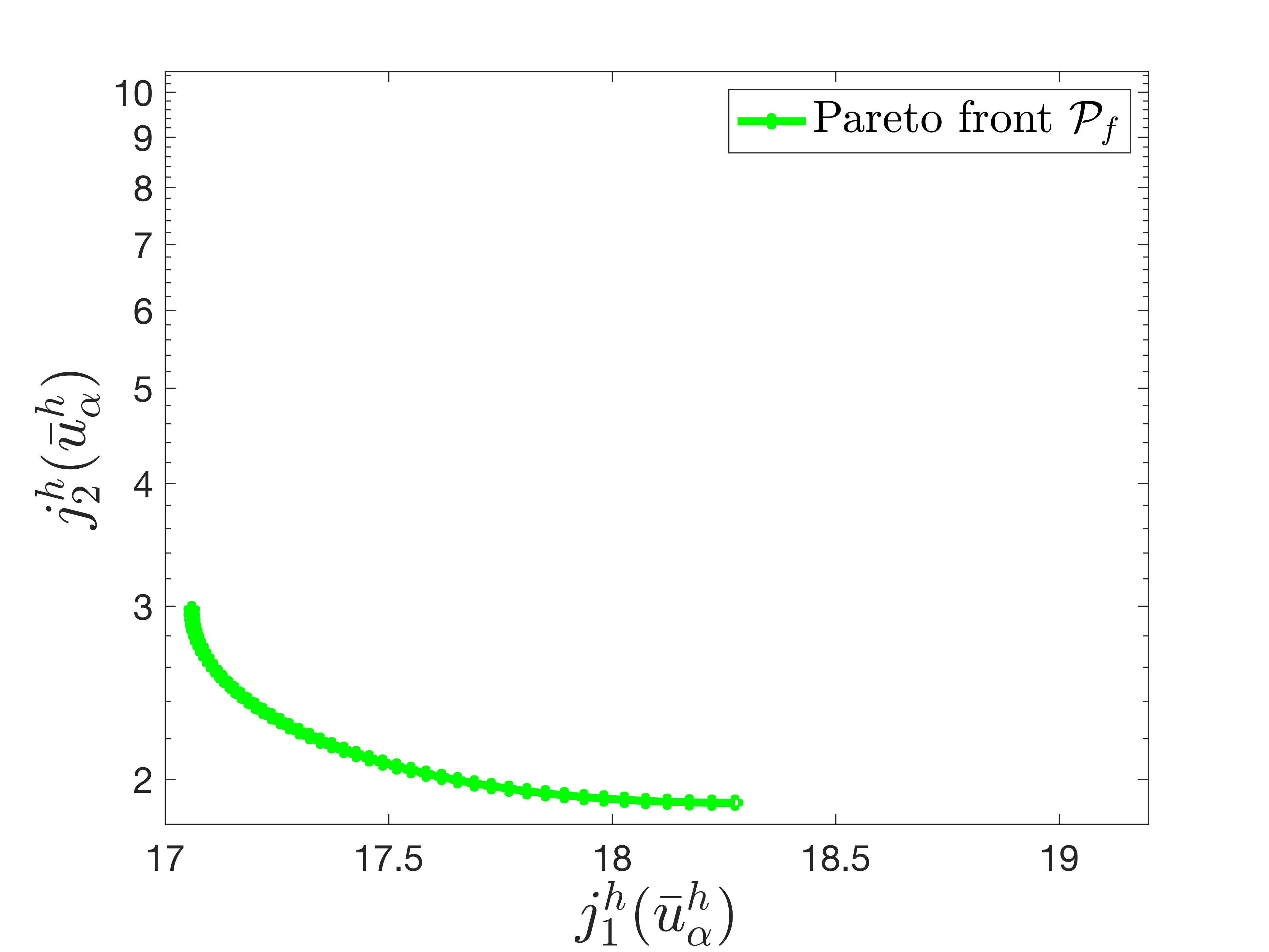}
\qquad
{\small{(1.G)}}
\end{minipage}
\begin{minipage}[c]{0.38\textwidth}\centering
\includegraphics[trim={0 0 0 0},clip,width=5.5cm,height=4.0cm,scale=0.1]{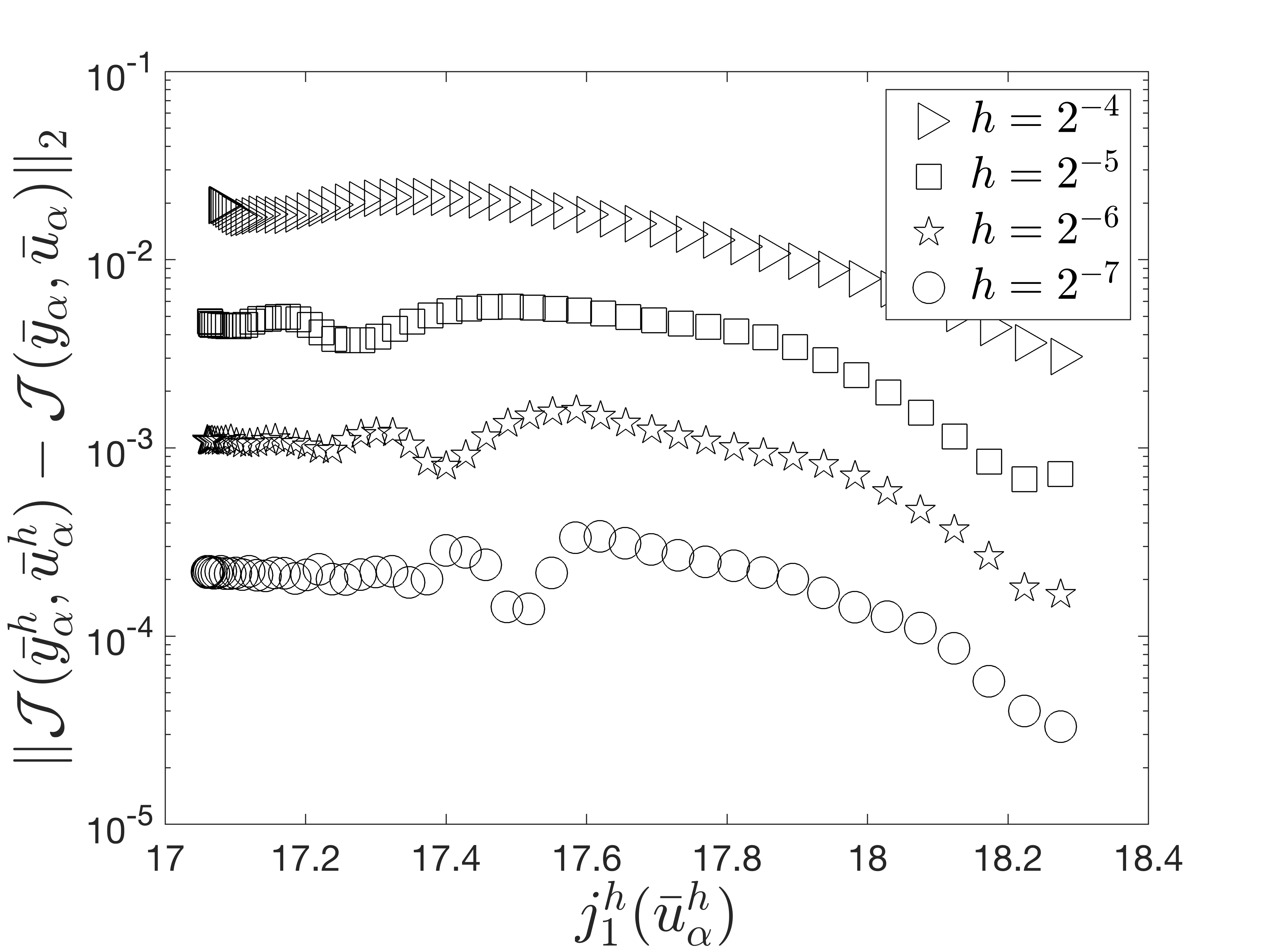}
\qquad
{\small{(1.H)}}
\end{minipage}
\caption{Pareto front and approximation error $\|\mathcal{J}(\bar{y}_{\alpha}^{h},\bar{u}_{\alpha}^{h}) - \mathcal{J}(\bar{y}_{\alpha},\bar{u}_{\alpha})\|_{2}$ with $50$ different values of $\alpha$ considering mesh refinement under the WSM for the cases $\lambda_1=\lambda_2=1$ (1.A)--(1.B), $\lambda_1=0.1$ and $\lambda_2=1$ (1.C)--(1.D), $\lambda_1=1$ and $\lambda_2=0.1$ (1.E)--(1.F), and  $\lambda_1=\lambda_2=0.1$ (1.G)--(1.H).}
\label{fig:ex1_1}
\end{center}
\end{figure}

\begin{figure}
\begin{center}
\begin{minipage}[c]{0.235\textwidth}\centering
\includegraphics[trim={0 0 0 0},clip,width=4.3cm,scale=0.1]{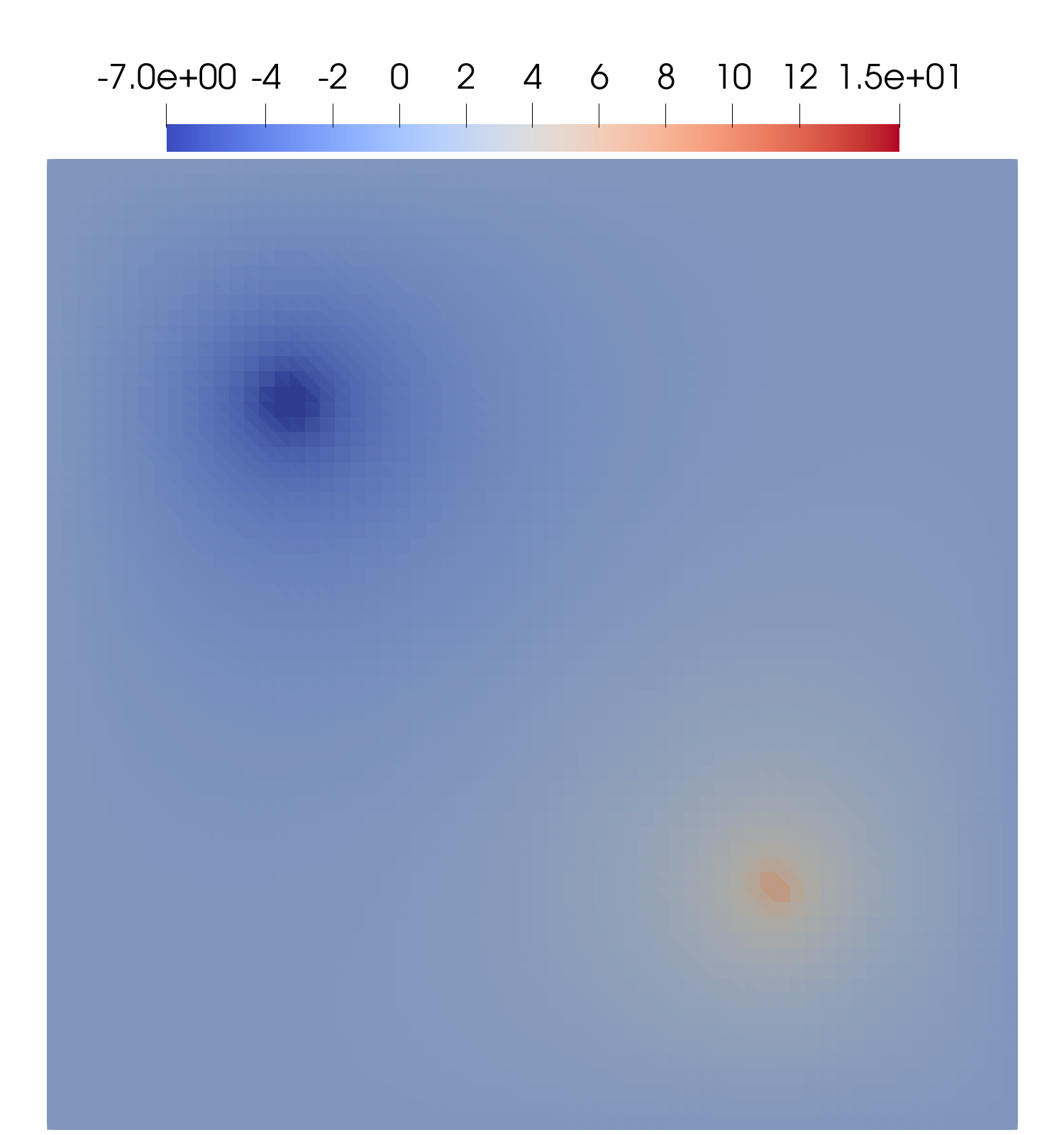}\\
\quad 
{\small{(2.A)}}
\end{minipage}
\begin{minipage}[c]{0.235\textwidth}\centering
\includegraphics[trim={0 0 0 0},clip,width=4.3cm,scale=0.1]{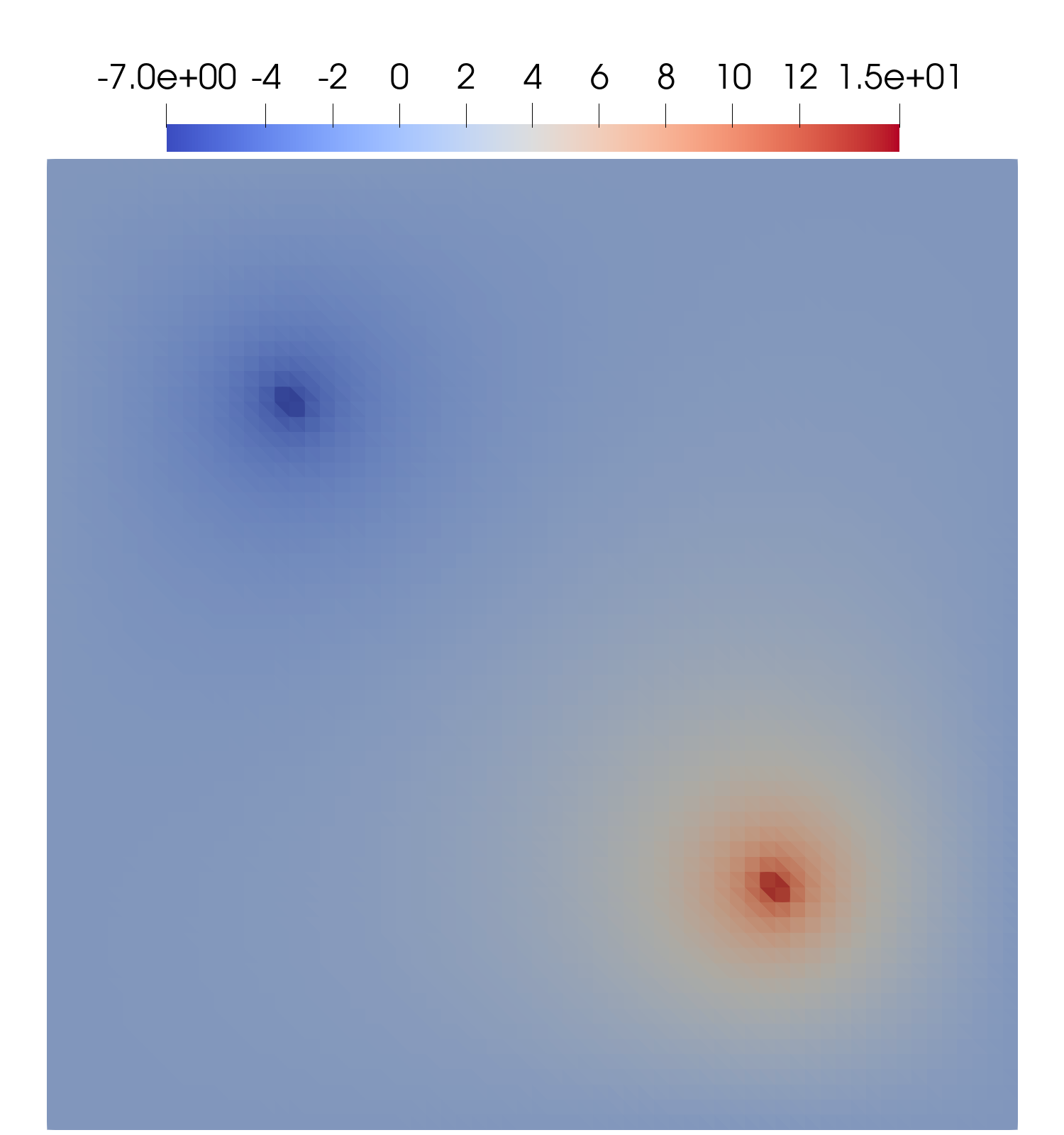}\\
\quad
{\small{(2.B)}}
\end{minipage}
\begin{minipage}[c]{0.235\textwidth}\centering
\includegraphics[trim={0 0 0 0},clip,width=4.3cm,scale=0.1]{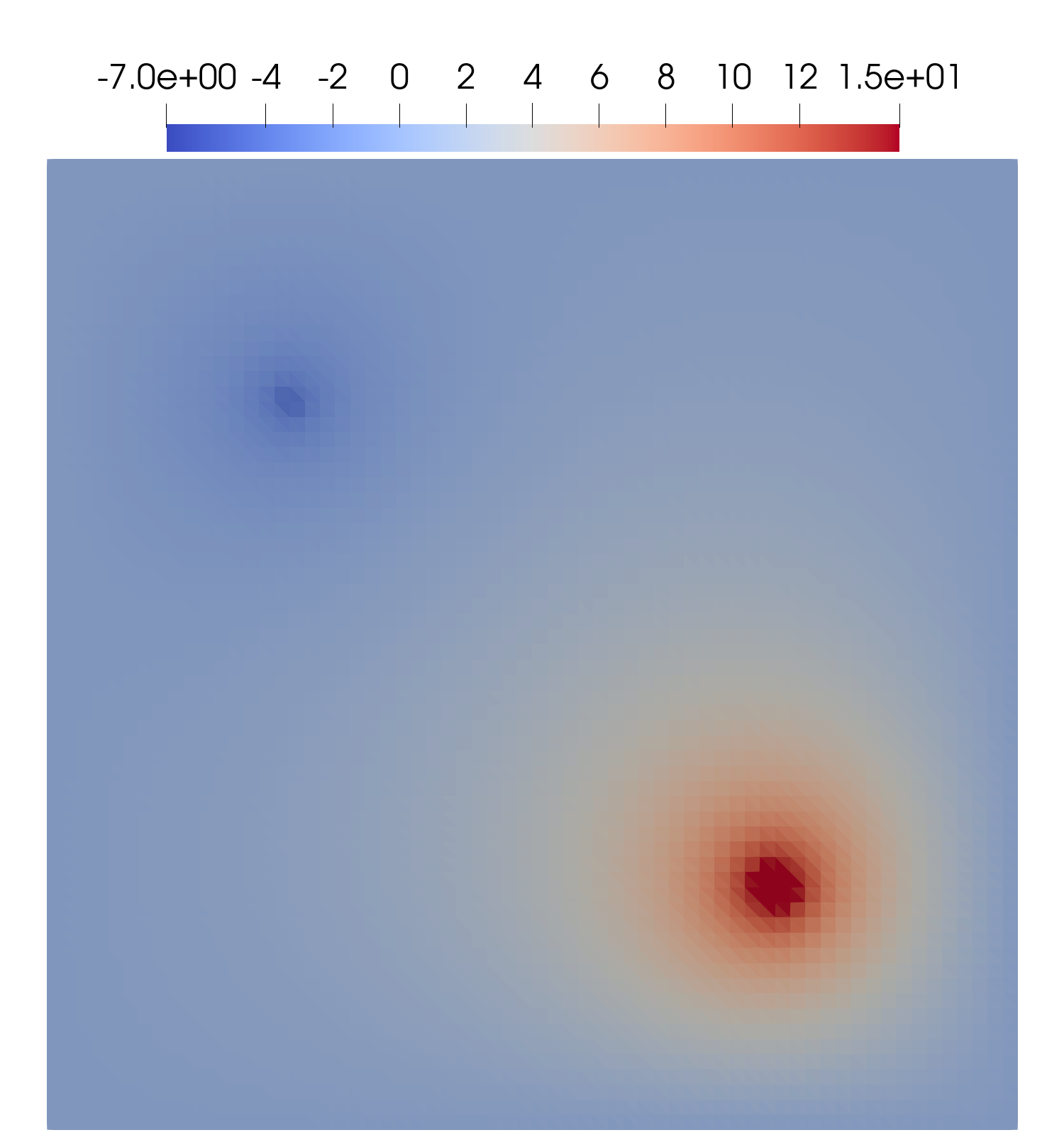}\\    
\quad
{\small{(2.C)}}
\end{minipage}
\begin{minipage}[c]{0.235\textwidth}\centering
\includegraphics[trim={0 0 0 0},clip,width=4.3cm,scale=0.1]{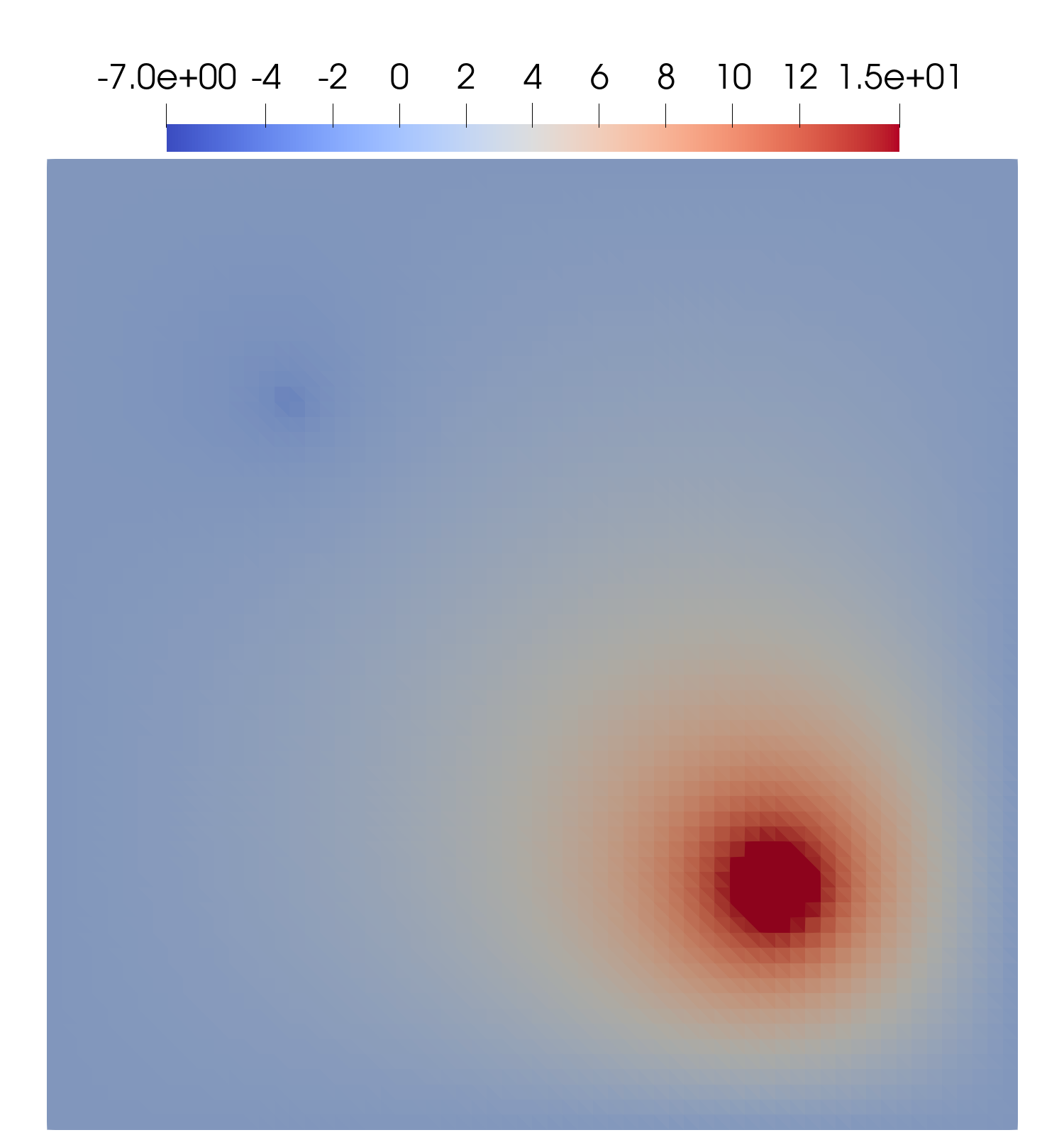}\\
\quad
{\small{(2.D)}}
\end{minipage}
\caption{Approximate optimal control $\bar{u}_{\alpha}^{h}$ when $\alpha=(0.2,0.8)$ (2.A), $\alpha=(0.4,0.6)$ (2.B), $\alpha=(0.6,0.4)$ (2.C), and $\alpha=(0.8,0.2)$ (2.D).
Here, $h=2^{-6}$ and $\lambda_1=\lambda_2=0.1$}
\label{fig:ex1_3}
\end{center}
\end{figure}

\begin{table}[h!]
\centering
\begin{tabular}{|c|c|c|c|c|}
\hline
$h$ & $\alpha=(0.2,0.8)$ & $\alpha=(0.4,0.6)$ &  $\alpha=(0.6,0.4)$ & $\alpha=(0.8,0.2)$ \\
\hline
$2^{-2}$ & 0.727125 & 0.994289 & 1.312741 & 1.580559 \\
$2^{-3}$ & 0.399550 & 0.555188 & 0.704527 & 0.790838 \\
$2^{-4}$ & 0.209558 & 0.300159 & 0.353305 & 0.389034 \\
$2^{-5}$ & 0.107604 & 0.155751 & 0.173634 & 0.193365 \\
\hline
Rate of convergence & 0.92 & 0.89 & 0.97 & 1.01 \\
\hline
\end{tabular}
\caption{Approximation error $\|\bar{u}_{\alpha} - \bar{u}_{\alpha}^{h}\|_{\U}$ considering the four different choices for $\alpha$.}
\label{table:ex1}
\end{table}


\subsection{Example 2 (RPM).}

For the reference point method, we choose in \eqref{def:zeta_1} the scaling parameters $\mathfrak{h}^{\perp} = \mathfrak{h}^{\|} = 0.2$.
We note that the values of the points $\zeta$ change in each refinement, due to their implicit dependency on $h$; see \eqref{def:general_zeta}.
In this example, we show the error in the steps $\zeta^{2}, \zeta^{4}, \zeta^{7}$ and $\zeta^{9}$, whose values on the finest mesh ($h=2^{-8}$) are
\begin{align}
    \label{Eq_zeta}
    \zeta^{2} \approx (16.89,2.58), \qquad 
    \zeta^{4} \approx (17.04,2.21), \qquad 
    \zeta^{7} \approx (17.49,1.82), \qquad 
    \zeta^{9} \approx (17.88,1.71).
\end{align}

We present the results obtained for this example in Figure \ref{fig:ex2_1} and Table \ref{table:ex2}.
In fact, we get similar results to those shown for the WSM. In particular, we observe optimal experimental rates of convergence $\mathcal{O}(h)$ for all the different cases of $\zeta$, which is in agreement with Theorem \ref{thm:improved_RPM}.

\begin{figure}
\begin{center}
\begin{minipage}[c]{0.235\textwidth}\centering
\includegraphics[trim={0 0 0 0},clip,width=4.3cm,scale=0.1]{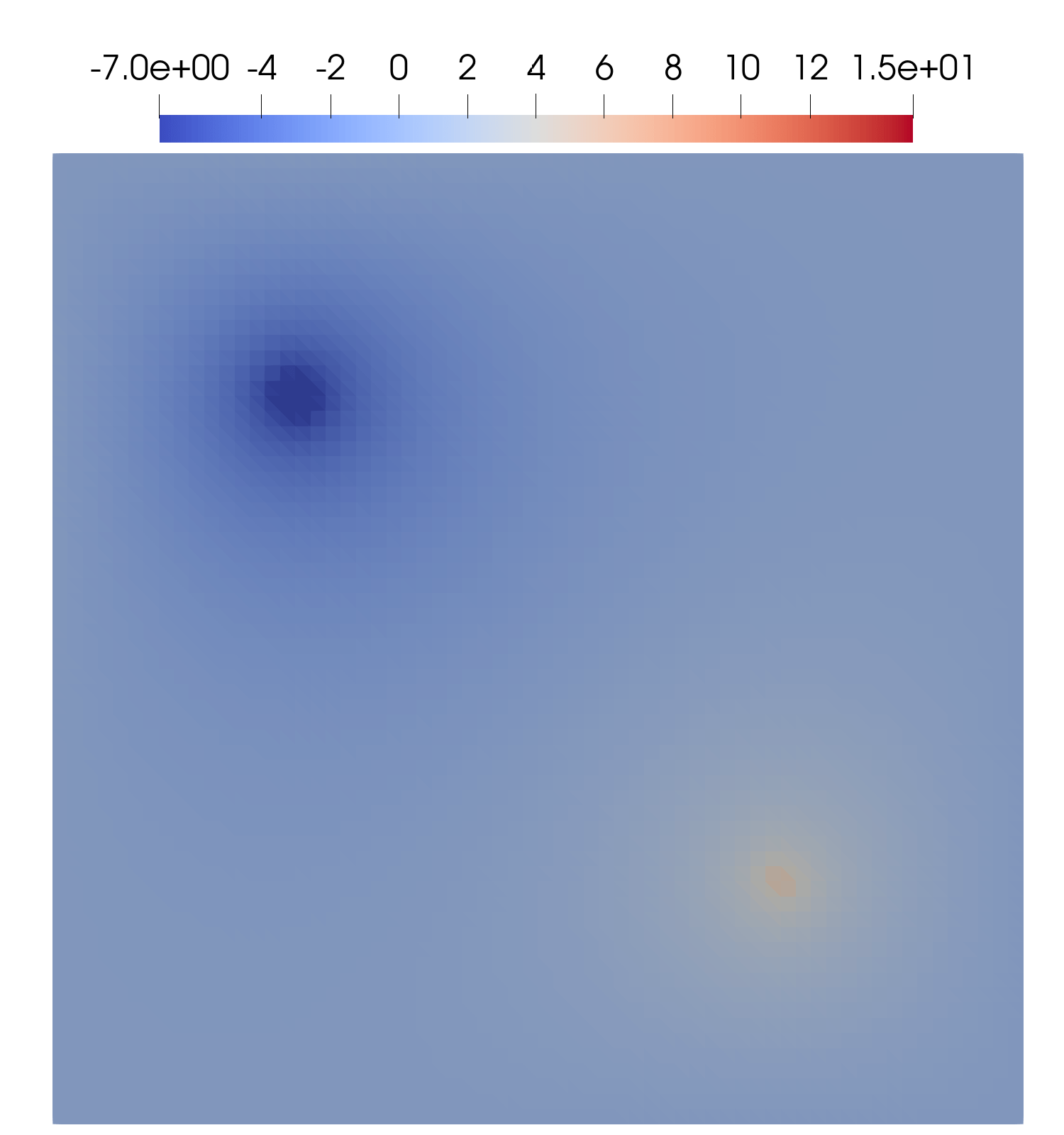}\\
\quad 
{\small{(3.A)}}
\end{minipage}
\begin{minipage}[c]{0.235\textwidth}\centering
\includegraphics[trim={0 0 0 0},clip,width=4.3cm,scale=0.1]{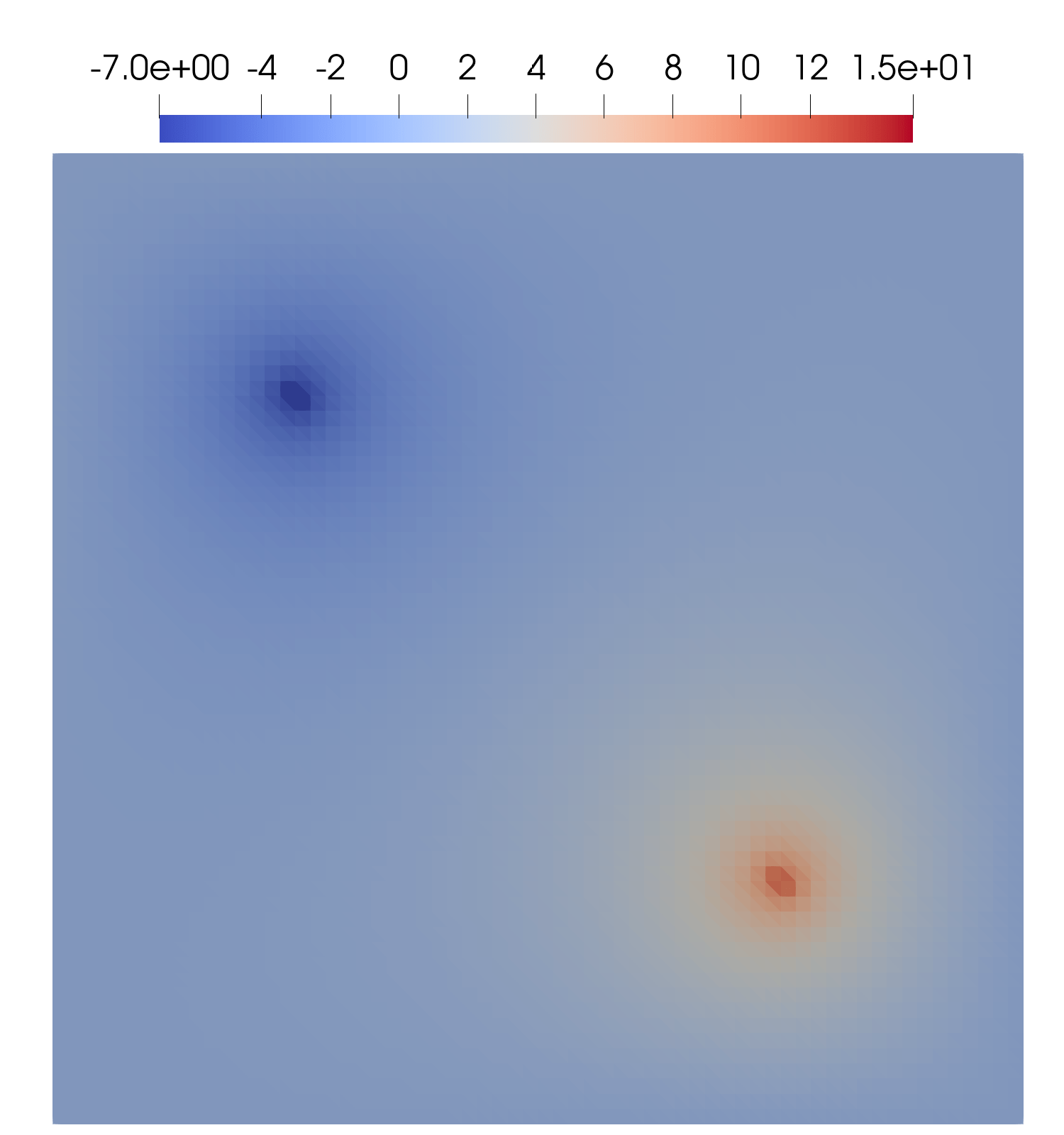}\\
\quad
{\small{(3.B)}}
\end{minipage}
\begin{minipage}[c]{0.235\textwidth}\centering
\includegraphics[trim={0 0 0 0},clip,width=4.3cm,scale=0.1]{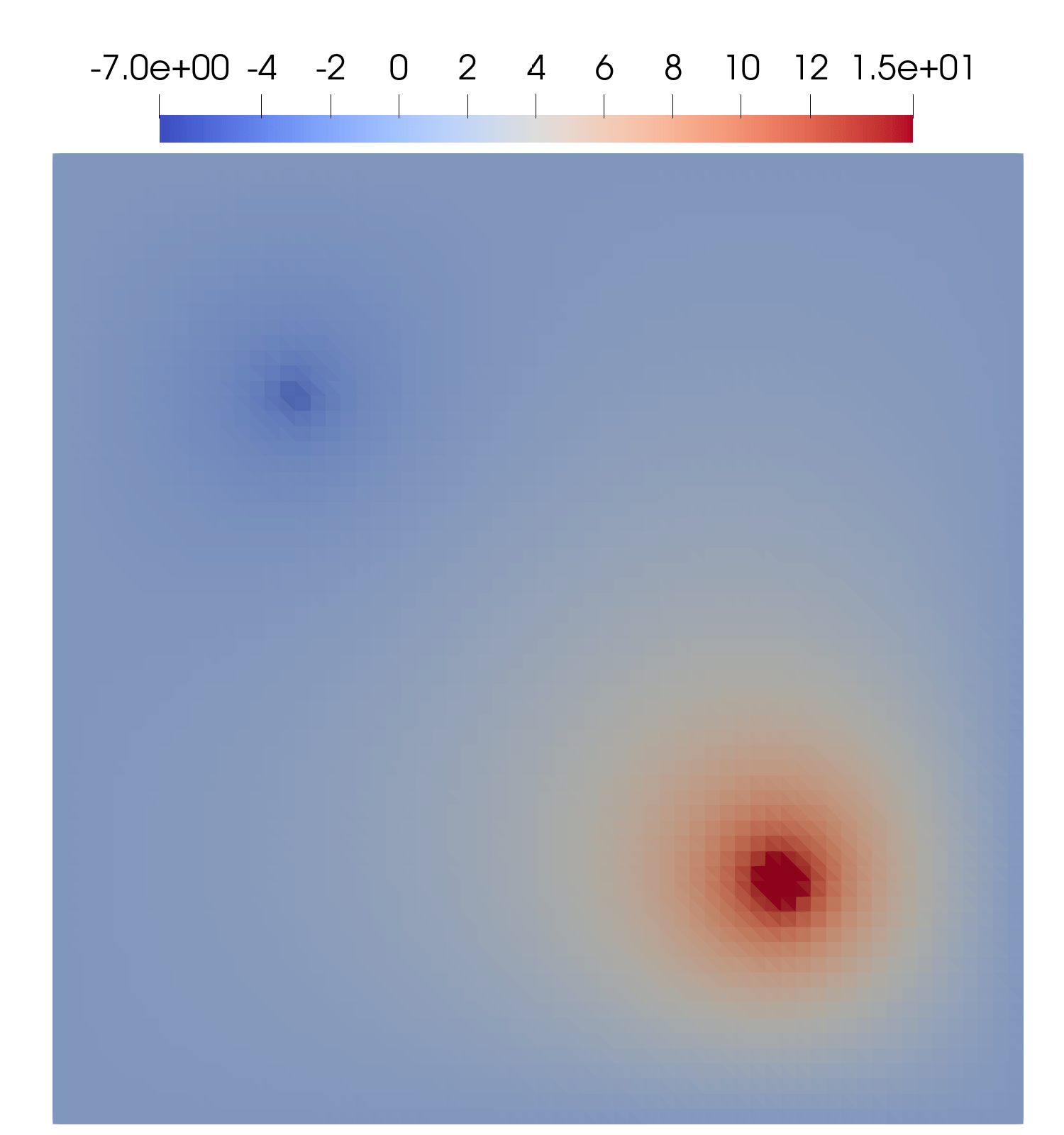}\\    
\quad
{\small{(3.C)}}
\end{minipage}
\begin{minipage}[c]{0.235\textwidth}\centering
\includegraphics[trim={0 0 0 0},clip,width=4.3cm,scale=0.1]{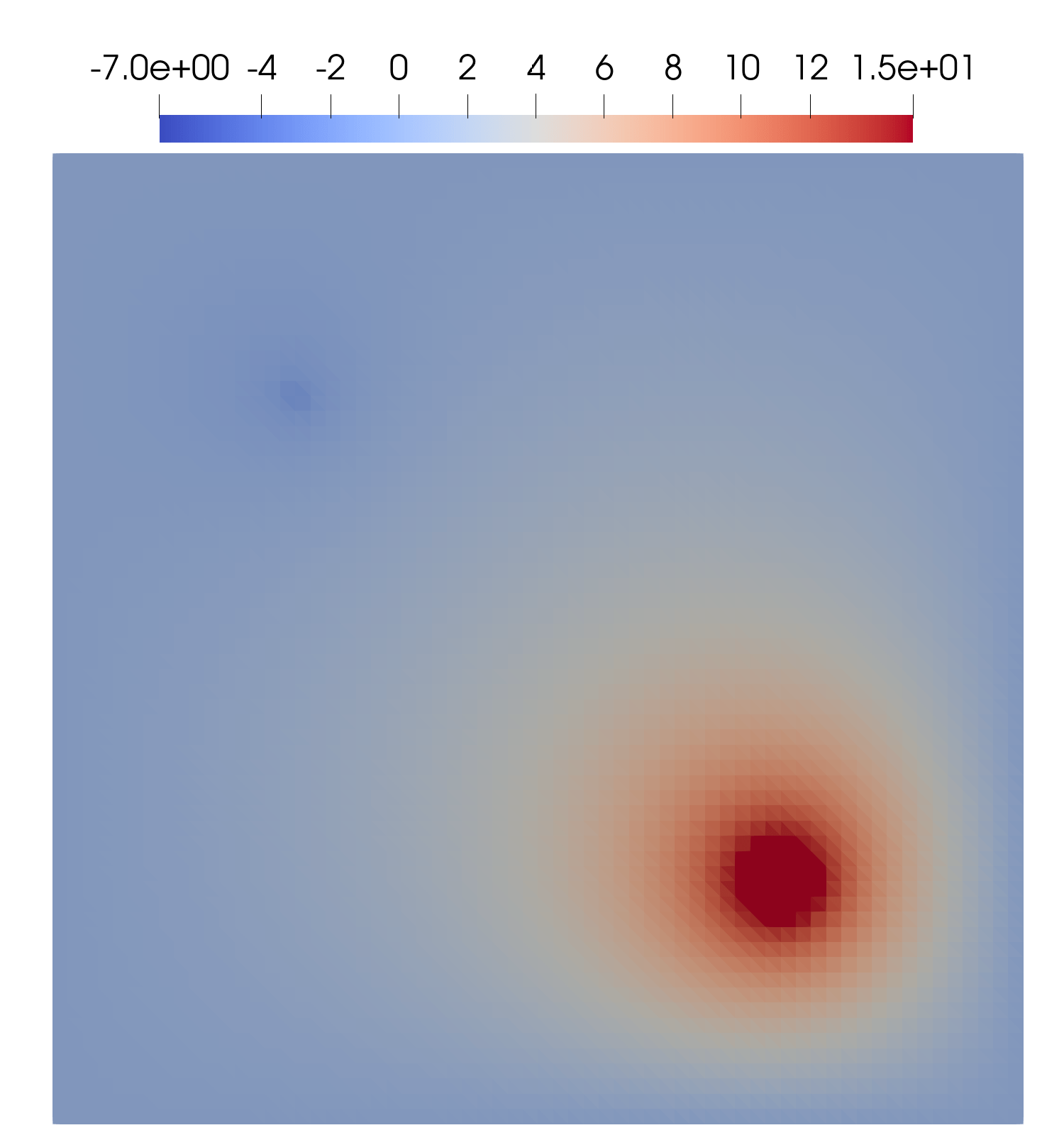}\\
\quad
{\small{(3.D)}}
\end{minipage}
\caption{Approximate optimal control $\bar{u}_{\zeta}^{h}$ for $\zeta^{9}$ (3.A), $\zeta^{7}$ (3.B), $\zeta^{4}$ (3.C), and $\zeta^{2}$ (3.D).
Here, $h=2^{-6}$ and $\lambda_1=\lambda_2=0.1$}
\label{fig:ex2_1}
\end{center}
\end{figure}

\begin{table}[h!]
\centering
\begin{tabular}{|c|c|c|c|c|}
\hline
$h$ & $\zeta^2$ & $\zeta^4$ &  $\zeta^7$ & $\zeta^9$ \\
\hline
$2^{-2}$ & 1.583765 & 1.338101 & 1.002442 & 0.928546 \\
$2^{-3}$ & 0.799229 & 0.711743 & 0.489412 & 0.375962 \\
$2^{-4}$ & 0.390953 & 0.353613 & 0.266316 & 0.195219 \\
$2^{-5}$ & 0.193401 & 0.173748 & 0.139464 & 0.096843 \\
\hline
Rate of convergence & 1.01 & 0.98 & 0.94 & 1.07 \\
\hline
\end{tabular}
\caption{Approximation error $\|\bar{u}_{\zeta} - \bar{u}_{\zeta}^{h}\|_{\U}$ considering the four different choices for $\zeta$ mentioned in \eqref{Eq_zeta}.}
\label{table:ex2}
\end{table}


\bibliographystyle{plainurl}
\bibliography{references.bib}
\end{document}